\documentclass[12pt]{article}
\usepackage{amsfonts, amsmath, amssymb, latexsym, eucal, amscd} 

\usepackage[dvips]{pict2e}

\usepackage{amsthm}
\usepackage{amscd}

\usepackage[dvipdfmx]{graphics}
\usepackage{cite}

\newtheorem{definition}{Definition}[section]
\newtheorem{theorem}[definition]{Theorem}
\newtheorem{lemma}[definition]{Lemma}
\newtheorem{corollary}[definition]{Corollary}

\newtheorem{example}[definition]{Example}

\newtheorem{problem}[definition]{Problem}
\newtheorem{note}[definition]{Note}
\newtheorem{assumption}[definition]{Assumption}
\newtheorem{proposition}[definition]{Proposition}

\begin{document} 

\title{\bf
Tridiagonal pairs of $q$-Racah type \\
and the
$q$-tetrahedron algebra
}
\author{
Paul Terwilliger}
\date{}

\maketitle
\begin{abstract}
Let $\mathbb F$ denote a field, and let $V$ denote a vector space over $\mathbb F$ with finite positive dimension.
We consider an ordered pair of $\mathbb F$-linear maps $A: V \to V$ and $A^*:V\to V$ such that
(i) each of $A,A^*$ is diagonalizable;
(ii) there exists an ordering $\lbrace V_i\rbrace_{i=0}^d$ of the  
eigenspaces of $A$ such that 
$A^* V_i \subseteq V_{i-1} + V_i+ V_{i+1}$ for $0 \leq i \leq d$, 
where $V_{-1} = 0$ and $V_{d+1}= 0$;
(iii) there exists an ordering $\lbrace V^*_i\rbrace_{i=0}^{\delta}$ of
the  
eigenspaces of $A^*$ such that
$A V^*_i \subseteq V^*_{i-1} + V^*_i+ V^*_{i+1} $ for
$0 \leq i \leq \delta $, 
where $V^*_{-1} = 0$ and $V^*_{\delta+1}= 0$;
(iv) there does not exist a subspace $U$ of $V$ such  that $AU\subseteq U$,
$A^*U \subseteq U$, $U\not=0$, $U\not=V$.
We call such a pair a tridiagonal pair on $V$. We assume that $A, A^*$ belongs to a family
of tridiagonal pairs said to have  $q$-Racah type.
There is an infinite-dimensional algebra $\boxtimes_q$
 called the $q$-tetrahedron algebra;
  it is generated by four copies
of $U_q(\mathfrak{sl}_2)$ that are related in a certain way. Using $A, A^*$ we construct two  $\boxtimes_q$-module structures on $V$.
 In this construction the two  main ingredients are the
double lowering map $\psi:V\to V$ due to Sarah Bockting-Conrad, and a certain invertible map $W:V\to V$
motivated by the spin model concept due to V. F. R. Jones.
\bigskip

 \end{abstract}
 
\section{Introduction}
\noindent  This paper is about a linear-algebraic object called a tridiagonal pair, and its
relationship to a certain infinite-dimensional algebra $\boxtimes_q$ called the $q$-tetrahedron algebra.
Before we explain our purpose in detail, we first define a tridiagonal pair. 
We will use the following terms.
Let $\mathbb F$ denote a field, and 
let $V$ denote a vector space over $\mathbb F$ with finite
positive dimension. 
Let ${\rm End}(V)$ denote the algebra consisting of the $\mathbb F$-linear maps from $V$ to $V$.
For $A \in {\rm End}(V)$
and a
subspace $U \subseteq V$,
we call $U$ an
 {\it eigenspace} of $A$ whenever 
 $U\not=0$ and there exists $\theta \in \mathbb F$ such that 
$U=\lbrace v \in V \vert Av = \theta v\rbrace$;
in this case $\theta$ is the {\it eigenvalue} of
$A$ associated with $U$.
We say that $A$ is {\it diagonalizable} whenever
$V$ is spanned by the eigenspaces of $A$.

\begin{definition}  \rm
(See \cite[Definition~1.1]{TD00}.)
\label{def:tdp}
Let $V$ denote a vector space over $\mathbb F$ with finite
positive dimension. 
By a {\it tridiagonal pair} (or {\it $TD$ pair})
on $V$,
we mean an ordered pair $A, A^*$ of elements in ${\rm End}(V)$ 
that satisfy the following four conditions.
\begin{enumerate}
\item Each of $A,A^*$ is diagonalizable.
\item There exists an ordering $\lbrace V_i\rbrace_{i=0}^d$ of the  
eigenspaces of $A$ such that 
\begin{equation}
A^* V_i \subseteq V_{i-1} + V_i+ V_{i+1} \qquad \qquad (0 \leq i \leq d),
\label{eq:t1}
\end{equation}
where $V_{-1} = 0$ and $V_{d+1}= 0$.
\item There exists an ordering $\lbrace V^*_i\rbrace_{i=0}^{\delta}$ of
the  
eigenspaces of $A^*$ such that 
\begin{equation}
A V^*_i \subseteq V^*_{i-1} + V^*_i+ V^*_{i+1} 
\qquad \qquad (0 \leq i \leq \delta),
\label{eq:t2}
\end{equation}
where $V^*_{-1} = 0$ and $V^*_{\delta+1}= 0$.
\item There does not exist a subspace $U$ of $V$ such  that $AU\subseteq U$,
$A^*U \subseteq U$, $U\not=0$, $U\not=V$.
\end{enumerate}
The TD pair $A,A^*$ is said to be {\it over $\mathbb F$}.
We call $V$ the 
{\it underlying
 vector space}.
\end{definition}

\begin{note}
\label{lem:convention}
\rm
According to a common notational convention, $A^*$ denotes 
the conjugate-transpose of $A$. We are not using this convention.
In a TD pair $A,A^*$ the linear maps  $A$ and $A^*$
are arbitrary subject to (i)--(iv) above.
\end{note}

\noindent We refer the reader to \cite{totBip} for background information on TD pairs.
In that article, the introduction summarizes the origin of the TD pair concept in algebraic graph theory, 
and Section 19 gives a comprehensive discussion of the current  state of the art.
\medskip

\noindent 
In order to motivate our results, we recall some basic facts about TD pairs.
Let $A,A^*$ denote a TD pair
on $V$, as in Definition 
\ref{def:tdp}. By 
\cite[Lemma 4.5]{TD00}
the integers $d$ and $\delta$ from
(ii), (iii) are equal; we call this
common value the {\it diameter} of the
pair. For $0 \leq i \leq d$ let $\theta_i$ (resp. $\theta^*_i$) denote
the eigenvalue of $A$ (resp. $A^*$) for the eigenspace $V_i$
(resp. $V^*_i$). By \cite[Theorem~11.1]{TD00} the scalars
\begin{align*}
\frac{\theta_{i-2}-\theta_{i+1}}{\theta_{i-1}-\theta_i},
\qquad \qquad
\frac{\theta^*_{i-2}-\theta^*_{i+1}}{\theta^*_{i-1}-\theta^*_i}
\end{align*}
are equal and independent of $i$ for $2 \leq i \leq d-1$. For
this recurrence the solutions can be given in closed form 
\cite[Theorem~11.2]{TD00}.
The ``most general'' solution is called $q$-Racah, and will
be described shortly. 
\medskip

\noindent By construction, the vector space $V$
has a direct sum decomposition into the eigenspaces 
$\lbrace V_i\rbrace_{i=0}^d$ 
of $A$
and the eigenspaces $\lbrace V^*_i\rbrace_{i=0}^d$ of $A^*$.
The vector space $V$ has
two more direct sum decompositions of interest,
called the first split decomposition
$\lbrace U_i\rbrace_{i=0}^d$ and
second split decomposition
$\lbrace U^\Downarrow_i\rbrace_{i=0}^d$.
By 
\cite[Theorem~4.6]{TD00}
the first split decomposition satisfies
\begin{align*}
&U_0 + U_1+ \cdots + U_i = V^*_0 + V^*_1 + \cdots + V^*_i,
\\
&U_i + U_{i+1}+ \cdots + U_d = V_i + V_{i+1} + \cdots + V_d
\end{align*}
for $0 \leq i \leq d$.
By \cite[Theorem~4.6]{TD00}
the second split decomposition satisfies
\begin{align*}
&U^\Downarrow_0 + U^\Downarrow_1+ \cdots + U^\Downarrow_i = V^*_0 + V^*_1 + \cdots + V^*_i,
\\
&U^\Downarrow_i + U^\Downarrow_{i+1}+ \cdots + U^\Downarrow_d =
V_0 + V_{1} + \cdots + V_{d-i}
\end{align*}
for $0 \leq i \leq d$.
By \cite[Theorem~4.6]{TD00},
\begin{align*}
(A-\theta_i I)U_i \subseteq U_{i+1},
\qquad \qquad (A^*-\theta^*_i I) U_i \subseteq U_{i-1},
\\
(A-\theta_{d-i} I)U^\Downarrow_i \subseteq U^\Downarrow_{i+1},
\qquad \qquad (A^*-\theta^*_i I) U^\Downarrow_i \subseteq U^\Downarrow_{i-1}
\end{align*}
for $0 \leq i \leq d$, where $U_{-1}=0$, $U_{d+1}=0$ 
and $U^\Downarrow_{-1}=0$, $U^\Downarrow_{d+1}=0$.
\medskip\

\noindent
We now describe the $q$-Racah case. The TD pair $A, A^*$ is said to have
{\it $q$-Racah type} whenever there exist nonzero $a,b,q \in \mathbb F$ such that $q^4 \not=1$ and
\begin{align*}
\theta_i = a q^{d-2i}+ a^{-1} q^{2i-d}, \qquad \qquad
\theta^*_i = b q^{d-2i}+ b^{-1} q^{2i-d}
\end{align*}
for $0 \leq i \leq d$. For the rest of this section, assume that $A, A^*$ has $q$-Racah type with $d\geq 1$.
\medskip

\noindent  In this paper, our purpose is to turn the vector space $V$ into a module for the $q$-tetrahedron algebra $\boxtimes_q$.
Over  the next few paragraphs, we will describe the maps that get used and explain what $\boxtimes_q$ is all about.
\medskip

\noindent 
We introduce an invertible $W \in {\rm End}(V)$  such that for $0 \leq i \leq d$, $V_i$ is an eigenspace of $W$ with
eigenvalue $(-1)^i a^i q^{i(d-i)}$. We remark that the idea behind $W$ goes back to the spin model concept introduced by V.~F.~R. Jones 
\cite{jones1}. In the interval since then, the idea was developed in the context of
association schemes \cite{bbj, j1, j2, anAlg},
distance-regular graphs \cite{bb,curtNom},
the subconstituent algebra \cite{caugh, curtThin},
spin Leonard pairs \cite{spinLP},
and  Leonard triples \cite{curtMod, totBip, terPseudo}. See \cite{spinModelNT} for a comprehensive description of $W$ in the context of 
spin models, distance-regular graphs, and spin Leonard pairs. We also remark that $W^2$ is closely related to 
the Lusztig automorphism of the $q$-Onsager algebra \cite{BK,terLusztig}; indeed $W^2=H$ where $H$ is from \cite[Section~3]{diagram}.
 In the present paper, we will obtain a number of identities involving $W^{\pm 1}$; for example
 \begin{align*}
 W&=  \sum_{i=0}^{d} \frac{ (-1)^i q^{i^2}(A-\theta_0 I )( A-\theta_{1}I ) \cdots (A-\theta_{i-1}I)}{(q^2;q^2)_i(a^{-1} q^{1-d}; q^2)_i},
 \\
W^{-1}&=  \sum_{i=0}^{d} \frac{ (-1)^i q^{-i^2}(A-\theta_0 I )( A-\theta_{1}I ) \cdots (A-\theta_{i-1}I)}{(q^{-2};q^{-2})_i(a q^{d-1}; q^{-2})_i}.
\end{align*}
For the above sums, the denominator notation is explained in Section 6.
\medskip

\noindent Next we recall the maps $K, B$. Following \cite[Section~1.1]{augmented},
we define $K, B \in {\rm End}(V)$ 
such that for $0 \leq i \leq d$, $U_i$ (resp. $U^\Downarrow_i$) is an eigenspace of
$K$ (resp. $B$) with eigenvalue $q^{d-2i}$. 
The maps $K, B$ are invertible.
By \cite[Section~1.1]{augmented},
\begin{align*}
&\frac{q K A - q^{-1} A K}{q-q^{-1}} = a K^2+ a^{-1} I,\qquad \qquad 
\frac{q B A - q^{-1} A B}{q-q^{-1}} =a^{-1} B^2+ a I.
\end{align*}
\noindent By \cite[Theorem~9.9]{bockting},
\begin{align*}
& aK^2 - \frac{a^{-1} q - a q^{-1}}{q-q^{-1}} K B - \frac{a q - a^{-1} q^{-1}}{q-q^{-1}} BK + a^{-1} B^2 = 0.
\end{align*}
We will show that
\begin{align*}
&A = a W^{-1} K W + a^{-1} W K^{-1} W^{-1},
\\
&\frac{q W^{-1} K W K^{-1} - q^{-1}  K^{-1} W^{-1} K W}{q-q^{-1} } = I,
\\
&\frac{q W^{-2} K W^2 K^{-1} - q^{-1}  K^{-1} W^{-2} KW^2}{q-q^{-1} } = I
\end{align*}
and also
\begin{align*}
&A = a^{-1} W^{-1} B W + a W B^{-1} W^{-1},
\\
&\frac{q W^{-1} B W B^{-1} - q^{-1}  B^{-1} W^{-1} B W}{q-q^{-1} } = I,
\\
&\frac{q W^{-2} B W^2 B^{-1} - q^{-1}  B^{-1} W^{-2} BW^2}{q-q^{-1} } = I.
\end{align*}

\noindent Next we discuss the maps  $M, N, Q$.  Following 
\cite[Section~6]{bocktingQexp} and \cite[Section~7]{diagram}
we define
\begin{align*}
M = \frac{a K - a^{-1} B}{a-a^{-1}},
\qquad \qquad
N = \frac{a^{-1} K^{-1} - a B^{-1} }{a^{-1} -a}.
\end{align*}
We will show that $W^{-1} M W = W N W^{-1}$; denote this common value by $Q$.
We will show that $Q$ is diagonalizable with eigenvalues $\lbrace q^{d-2i} \rbrace_{i=0}^d$; in particular $Q$ is invertible.
\medskip

\noindent Next we recall the double lowering map $\psi$, due to Sarah Bockting-Conrad \cite{bockting1}.
By \cite[Lemma~11.2, Corollary~15.3]{bockting1},
\begin{align*} 
\psi U_i \subseteq U_{i-1}, \qquad \qquad 
\psi U^\Downarrow_i \subseteq U^\Downarrow_{i-1} \qquad \qquad (0 \leq i \leq d).
\end{align*}
By \cite[Theorem~9.8]{bockting}, $\psi$ is equal to each of
\begin{align*}
&
\frac{ I - BK^{-1}}{q(a I - a^{-1} B K^{-1})}, 
\qquad \qquad 
\frac{ I - KB^{-1}}{q(a^{-1} I - a K B^{-1})},
\\
&
\frac{ q(I-K^{-1}B)}{a I - a^{-1} K^{-1}B}, \qquad \qquad \quad \,
\frac{q(I-B^{-1}K)}{a^{-1} I - a B^{-1} K}.
\end{align*}
Expanding on \cite[Lemma~6.8]{bocktingQexp}, we will show that
\begin{align*}
\psi + \frac{q  A M^{-1} - q^{-1} M^{-1} A}{q^2-q^{-2}} &= \frac{a+a^{-1}}{q+q^{-1}} I,
\\
\psi + \frac{q N^{-1} A - q^{-1} AN^{-1} }{q^2-q^{-2}}  &= \frac{a+a^{-1}}{q+q^{-1}} I.
\end{align*}
\noindent We will also show that $A$ commutes with $\psi-Q^{-1}$, and that
\begin{align*}
&W \psi W^{-1} + Q^{-1} = \psi + M^{-1},
\qquad \qquad
W^{-1} \psi W + Q^{-1} = \psi + N^{-1}.
\end{align*}

\noindent Next we recall the Casimir element $\Lambda$. Following \cite[Lemma~7.2]{bockting} we define
\begin{align*}
 \Lambda=  \psi (A - a K - a^{-1} K^{-1}) + q^{-1} K + q K^{-1}.
\end{align*}
It is shown in  \cite[Lemmas~7.3, 8.3, 9.1]{bockting} 
 that $\Lambda$ commutes with each of
$A, K, B, \psi$. By this and the construction, $\Lambda$ commutes with
$W, M,N,Q$.
We will show that
\begin{align*}
M^{-1} + \frac{q \psi A - q^{-1} A \psi}{q^2-q^{-2}}  &= \frac{\Lambda}{q+q^{-1}},
\\
N^{-1}+ \frac{q  A\psi - q^{-1}  \psi A}{q^2-q^{-2}} &= \frac{\Lambda}{q+q^{-1}}.
\end{align*}
We will also show that
\begin{align*}
(\psi - Q^{-1})((q+q^{-1})I-A) = (a+a^{-1})I - \Lambda.
\end{align*}

\noindent Next we recall the $q$-tetrahedron algebra $\boxtimes_q$. This infinite-dimensional algebra was introduced in
\cite{qtet}, and used 
to study the TD pairs of $q$-geometric type. See 
\cite{qinv, ItoTerwilliger1, evalTetq, miki, pospart, yy}
for subsequent work. The algebra $\boxtimes_q$ is defined by generators and relations. To describe the
generators, let $\mathbb Z_4 = \mathbb Z / 4 \mathbb Z$ denote the cyclic group of order 4. 
The algebra $\boxtimes_q$ has eight generators
\begin{align*}
\lbrace x_{ij}\;|\;i,j \in 
\mathbb Z_4,
\;\;j-i=1 \;\mbox{\rm {or}}\;
j-i=2\rbrace.
\end{align*}
The defining relations for $\boxtimes_q$ are given in Definition \ref{def:qtet}
below. We introduce a type of $\boxtimes_q$-module, said to be $t$-segregated; here $t$ is a nonzero scalar parameter.  We will show that
on a $t$-segregated $\boxtimes_q$-module, the following four elements coincide and commute with everything in $\boxtimes_q$:
\begin{align*}
&
 t(x_{01}x_{23}-1)+qx_{30}+q^{-1}x_{12},
 \qquad 
  t^{-1}(x_{12}x_{30}-1)+qx_{01}+q^{-1}x_{23},
\\
&
 t(x_{23}x_{01}-1)+qx_{12}+q^{-1}x_{30},
 \qquad 
 t^{-1}(x_{30}x_{12}-1)+qx_{23}+q^{-1}x_{01}.
\end{align*}
Let $\Upsilon$ denote the common value of the above four elements.
\medskip

\noindent We now describe our two main results. In this description, we refer to the above TD pair $A,A^*$ on $V$ that has $q$-Racah type.
In our first main result, we show  that $V$
becomes an $a$-segregated
 $\boxtimes_q$-module on which the $\boxtimes_q$-generators act as follows:
\bigskip

\centerline{
\begin{tabular}[t]{c|cccc}
   {\rm generator} & $x_{01}$ & $x_{12}$ & $x_{23}$ & $x_{30}$
\\
\hline
{\rm action on $V$}  & 
$W^{-1} K W$ & $W K^{-1} W^{-1} $ & $Q^{-1}+W \psi W^{-1} $
&$Q^{-1}+ W^{-1} \psi W $
\\
\end{tabular}}
\bigskip

\centerline{
\begin{tabular}[t]{c|cccc}
   {\rm generator} & $x_{02}$ & $x_{13}$ & $x_{20}$ & $x_{31}$
\\
\hline
{\rm action on $V$}  & 
 $Q$ & $K^{-1}$ & $Q^{-1}$ & $K$
\\
\end{tabular}}
\bigskip

\noindent Moreover $\Upsilon=\Lambda$ on $V$. In our second main result, we show that $V$ becomes an $a^{-1}$-segregated
 $\boxtimes_q$-module on which the $\boxtimes_q$-generators act as follows:
\bigskip

\centerline{
\begin{tabular}[t]{c|cccc}
   {\rm generator} & $x_{01}$ & $x_{12}$ & $x_{23}$ & $x_{30}$
\\
\hline
{\rm action on $V$}  & 
$W^{-1} B W$ & $W B^{-1} W^{-1} $ & $Q^{-1} + W \psi W^{-1} $
&$Q^{-1} + W^{-1} \psi W $
\\
\end{tabular}}
\bigskip

\centerline{
\begin{tabular}[t]{c|cccc}
   {\rm generator} & $x_{02}$ & $x_{13}$ & $x_{20}$ & $x_{31}$
\\
\hline
{\rm action on $V$}  & 
 $Q$ & $B^{-1}$ & $Q^{-1}$ & $B$
\\
\end{tabular}}
\bigskip

\noindent Moreover $\Upsilon=\Lambda$ on $V$.
\medskip

\noindent This paper is organized as follows. In Section 2 we recall the notion of a tridiagonal system.
In Section 3 we recall the $q$-Dolan/Grady relations and discuss their basic properties.
In Section 4 we introduce a certain map that makes it easier to discuss elements in ${\rm End}(V)$ that commute with $A$.
In Section 5 we introduce the element $W$, and 
in Section 6 we display some identities involving $W^{\pm 1}$.
In Section 7 we discuss the elements $K, B$ and in
Section 8 we discuss the elements $M,N, Q$.
In Section 9 we describe how $W, K, B, Q$ are related using the concept of an equitable triple.
In Section 10 we discuss the double lowering map $\psi$,
and in Section 11 we discuss the Casimir element $\Lambda$.
In Section 12 we describe the element $\psi-Q$ in some detail.
In Section 13 we discuss how $W, K$ are related and how $W, B$ are related.
In Section 14 we recall the algebra $U_q(\mathfrak{sl}_2)$ in its equitable presentation.
In Section 15 we recall the $q$-tetrahedron algebra $\boxtimes_q$,
and in Sections 16, 17 we discuss the $t$-segregated $\boxtimes_q$-modules.
In Section 18 we give our main results, which are Theorems
\ref{thm:one},
\ref{thm:two}.
In Section 19 we give some suggestions for future research.

\section{Tridiagonal systems}

\noindent
We now begin our formal argument. When working with a TD pair, it is often convenient to consider
a closely related object called a TD system. We will review this
notion after some notational comments.
Recall the natural numbers $\mathbb N = \lbrace 0,1,2,\ldots\rbrace$ and integers
$\mathbb Z = \lbrace 0, \pm 1, \pm 2,\ldots \rbrace$.
Recall the field $\mathbb F$. Every vector space discussed in this paper is over $\mathbb F$.
 Every algebra discussed in this paper is associative, over $\mathbb F$, and has a multiplicative identity.
 A subalgebra has the same multiplicative identity as the parent algebra.
For the rest of this paper, $V$ denotes a vector space  with finite
positive dimension.
Recall the algebra  ${\rm End}(V)$ from above Definition \ref{def:tdp}.
Let $A$ denote a diagonalizable element in ${\rm End}(V)$.
Let $\lbrace V_i\rbrace_{i=0}^d$ denote an ordering of the eigenspaces of $A$.
For $0\leq i \leq d$ let $\theta_i$ denote the eigenvalue of $A$ for $V_i$.
Define $E_i \in 
{\rm End}(V)$ 
such that $(E_i-I)V_i=0$ and $E_iV_j=0$ for $j \neq i$ $(0 \leq j \leq d)$.
Thus $E_i$ is the projection from $V$ onto $V_i$.
We call $E_i$ the {\it primitive idempotent} of $A$ corresponding to $V_i$
(or $\theta_i$).
Observe that
(i) $V_i = E_iV$ $(0 \leq i \leq d)$;
(ii) $E_iE_j=\delta_{i,j}E_i$ $(0 \leq i,j \leq d)$;
(iii) $I=\sum_{i=0}^d E_i$;
(iv) $A=\sum_{i=0}^d \theta_i E_i$;
(v) $AE_i = \theta_i E_i = E_i A$ $(0 \leq i \leq d)$.
Moreover
\begin{equation}         \label{eq:defEi}
  E_i=\prod_{\stackrel{0 \leq j \leq d}{j \neq i}}
          \frac{A-\theta_jI}{\theta_i-\theta_j} \qquad \qquad (0 \leq i \leq d).
\end{equation}
Let $\mathcal D$ denote the subalgebra of ${\rm End}(V)$ generated by $A$.
Note that 
$\lbrace A^i\rbrace_{i=0}^d$
is a basis for the vector space $\mathcal D$, 
and  $\prod_{i=0}^d(A-\theta_iI)=0$. Moreover
$\lbrace E_i\rbrace_{i=0}^d$ is a basis for the vector space $\mathcal D$.
Now let $A,A^*$ denote a TD pair on $V$, as in Definition 
\ref{def:tdp}. An ordering of the eigenspaces of $A$ (resp. $A^*$)
is said to be {\it standard} whenever it satisfies 
(\ref{eq:t1})
 (resp. (\ref{eq:t2})). 
We comment on the uniqueness of the standard ordering.
Let $\lbrace V_i\rbrace_{i=0}^d$ denote a standard ordering of the eigenspaces of $A$.
By \cite[Lemma~2.4]{TD00}, 
 the ordering $\lbrace V_{d-i}\rbrace_{i=0}^d$ is also standard and no further
 ordering
is standard.
A similar result holds for the eigenspaces of $A^*$.
An ordering of the primitive idempotents 
 of $A$ (resp. $A^*$)
is said to be {\it standard} whenever
the corresponding ordering of the eigenspaces of $A$ (resp. $A^*$)
is standard.

\begin{definition} \rm (See  \cite[Definition~2.1]{TD00}.)
 \label{def:TDsystem} 
By a {\it tridiagonal system} (or {\it  $TD$ system}) on $V$ we mean a sequence
\begin{align*}
 \Phi=(A;\{E_i\}_{i=0}^d;A^*;\{E^*_i\}_{i=0}^d)
\end{align*}
that satisfies (i)--(iii) below:
\begin{enumerate}
\item[\rm (i)]
$A,A^*$ is a TD pair on $V$;
\item[\rm (ii)]
$\lbrace E_i\rbrace_{i=0}^d$ is a standard ordering
of the primitive idempotents of $A$;
\item[\rm (iii)]
$\lbrace E^*_i\rbrace _{i=0}^d$ is a standard ordering
of the primitive idempotents of $A^*$.
\end{enumerate}
The TD system $\Phi$ is said to be {\it over} $\mathbb F$.
We call $V$ the {\it underlying vector space}.
\end{definition}


\noindent 
Let $\Phi=(A;\lbrace E_i\rbrace_{i=0}^d;A^*; \lbrace E^*_i\rbrace_{i=0}^d)$ 
denote a TD system on $V$. Then
the following is a TD system on $V$:
\begin{align*}
\Phi^\Downarrow&=(A;\lbrace E_{d-i}\rbrace_{i=0}^d;A^*; \lbrace E^*_i\rbrace_{i=0}^d).
\end{align*}
\noindent For any object $f$ attached to $\Phi$,
let $f^\Downarrow$ denote the corresponding object attached to
 $\Phi^{\Downarrow}$.

\begin{definition}        \label{def}
\rm
Let $\Phi=(A;\lbrace E_i\rbrace_{i=0}^d;A^*; \lbrace E^*_i\rbrace_{i=0}^d)$ 
denote a TD system on $V$.
For $0 \leq i \leq d$ let $\theta_i$ (resp. $\theta^*_i$)
denote the eigenvalue of $A$ (resp. $A^*$)
for the eigenspace $E_iV$ (resp. $E^*_iV$).
We call $\lbrace \theta_i\rbrace_{i=0}^d$ (resp. $\lbrace \theta^*_i\rbrace_{i=0}^d$)
the {\it eigenvalue sequence}
(resp. {\it dual eigenvalue sequence}) of $\Phi$.
\end{definition}
\noindent Referring to Definition
\ref{def}, 
  we emphasize that $\lbrace \theta_i\rbrace _{i=0}^d$ are mutually distinct, and 
$\lbrace \theta^*_i\rbrace_{i=0}^d$ are
mutually distinct. By 
\cite[Theorem~11.1]{TD00}  the expressions
\begin{align*}
\frac{\theta_{i-2}-\theta_{i+1}}{\theta_{i-1}-\theta_i},\qquad \qquad
\frac{\theta^*_{i-2}-\theta^*_{i+1}}{\theta^*_{i-1}-\theta^*_i}
\end{align*}
are equal and independent of $i$ for $2 \leq i \leq d-1$. For this recurrence the 
solutions can be expressed in closed form
\cite[Theorem~11.2]{TD00}.
 The ``most general'' solution is called $q$-Racah, and described below.

\begin{definition}
\label{def:qrac} \rm
Let $\Phi$ denote a TD system on $V$, with eigenvalue sequence
$\lbrace \theta_i \rbrace_{i=0}^d$ and dual eigenvalue sequence
$\lbrace \theta^*_i\rbrace_{i=0}^d $. Then $\Phi$ is said to have
{\it $q$-Racah type} whenever there exist nonzero $a, b, q \in \mathbb F$ such that $q^4 \not=1$ and
\begin{align}
\theta_i = a q^{d-2i} + a^{-1} q^{2i-d}, \qquad \qquad 
\theta^*_i = b q^{d-2i} + b^{-1} q^{2i-d}
\label{eq:eigform}
\end{align}
for $0 \leq i \leq d$.
\end{definition}
\noindent From now until the end of Section 13, we fix a TD system 
$\Phi=(A;\lbrace E_i\rbrace_{i=0}^d;A^*; \lbrace E^*_i\rbrace_{i=0}^d)$ on $V$
that has $q$-Racah type, with eigenvalue sequence $\lbrace \theta_i \rbrace_{i=0}^d$ and dual eigenvalue
sequence 
 $\lbrace \theta^*_i \rbrace_{i=0}^d$
as in Definition
\ref{def:qrac}. To avoid trivialities we assume that $d\geq 1$. Let $\mathcal D$ denote the subalgebra
of ${\rm End}(V)$ generated by $A$.
\medskip

\noindent
We mention some basic results for later use.

\begin{lemma} The following hold:
\begin{enumerate}
\item[\rm (i)] $q^{2i}\not=1$ for $1 \leq i \leq d$;
\item[\rm (ii)] neither of $a^2, b^2$ is among $q^{2d-2}, q^{2d-4}, \ldots, q^{2-2d}$.
\end{enumerate}
\end{lemma}
\begin{proof}
Use the  sentence below Definition \ref{def}, along with Definition \ref{def:qrac}.
\end{proof}

\begin{lemma}
\label{lem:onefactor}
For $1 \leq i \leq d$,
\begin{align*}
\frac{ q \theta_{i-1} - q^{-1} \theta_{i}}{q^2-q^{-2}} = a q^{d-2i+1},
\qquad \qquad 
\frac{ q \theta_i- q^{-1} \theta_{i-1}}{q^2-q^{-2}} = a^{-1} q^{2i-d-1}.
\end{align*}
\end{lemma}
\begin{proof}  By the form of the eigenvalue expressions in (\ref{eq:eigform}).
\end{proof}

\begin{lemma}
\label{lem:2factors}
For $0 \leq i,j\leq d$ such that $\vert i-j \vert = 1$,
\begin{align}
\frac{ q \theta_i - q^{-1} \theta_j}{q^2-q^{-2}} \,
\frac{ q \theta_j - q^{-1} \theta_i}{q^2-q^{-2}} = 1.
\label{eq:2factors}
\end{align}
\end{lemma}
\begin{proof} By Lemma
\ref{lem:onefactor}.
\end{proof}
 
 \noindent Let $X \in {\rm End}(V)$.  Using $I=E_0+ \cdots + E_d$ we have
 \begin{align*}
 X = I X  I = \sum_{i=0}^d \sum_{j=0}^d E_i X E_j.
 \end{align*}
 With this result we routinely obtain the following three lemmas.

 \begin{lemma}
 \label{lem:Zero}
 For $X \in {\rm End}(V)$ the following are equivalent:
 \begin{enumerate}
 \item[\rm (i)] $E_i  X E_j = 0 $ for $0 \leq i,j\leq d$;
 \item[\rm (ii)] $X=0$.
 \end{enumerate}
 \end{lemma}

 \begin{lemma}
 \label{lem:One}
 For $X \in {\rm End}(V)$ the following are equivalent:
 \begin{enumerate} 
 \item[\rm (i)] $E_i  X E_j = 0 $ if $i\not=j$ $(0 \leq i,j\leq d)$;
 \item[\rm (ii)] $X E_iV \subseteq E_iV$ for $0 \leq i \leq d$;
  \item[\rm (iii)] $A$ commutes with $X$.
 \end{enumerate}
 \end{lemma}

\begin{lemma}
\label{lem:four1}
For $X \in {\rm End}(V)$ the following are equivalent:
\begin{enumerate}
\item[\rm (i)] $E_i X E_j = 0$ if $\vert i-j \vert > 1$ $(0 \leq i,j\leq d)$;
\item[\rm (ii)] for $0 \leq i \leq d$,
\begin{align*}
X E_iV \subseteq E_{i-1}V + E_iV + E_{i+1}V,
\end{align*}
where $E_{-1}=0$ and $E_{d+1}=0$.
\end{enumerate}
\end{lemma} 

\begin{definition}\label{def:tridiag}\rm
Referring to Lemma \ref{lem:four1}, we say that {\it $X$ acts on the eigenspaces of $A$ in a tridiagonal fashion} whenever
the equivalent conditions (i), (ii) hold.
\end{definition}

\begin{example} \rm The elements $I, A, A^*$ act on the eigenspaces of $A$ in a tridiagonal fashion. 
\end{example}

 \section{The $q$-Dolan/Grady relations}
 
 We continue to discuss the TD system $\Phi=(A;\lbrace E_i\rbrace_{i=0}^d;A^*; \lbrace E^*_i\rbrace_{i=0}^d)$ on $V$ that has $q$-Racah type.
 In this section we consider how $A$ and $A^*$ are related.
  Recall the notation
\begin{align*}
\lbrack n \rbrack_q = \frac{q^n-q^{-n}}{q-q^{-1}} \qquad \qquad n \in \mathbb Z.
\end{align*}
For elements $X$, $Y$ in any algebra, their
commutator and $q$-commutator are given by
\begin{align*}
\lbrack X, Y\rbrack = XY-YX, \qquad \qquad
\lbrack X, Y\rbrack_q = q XY - q^{-1} YX.
\end{align*}
Note that
\begin{align*}
\lbrack X, \lbrack X, Y\rbrack_q \rbrack_{q^{-1}}  &=  X^2 Y - (q^2+q^{-2})  X Y X +  Y X^2,
\\
\lbrack X, \lbrack X, \lbrack X, Y\rbrack_q \rbrack_{q^{-1}} \rbrack &=  X^3 Y - \lbrack 3 \rbrack_q X^2 Y X + \lbrack 3 \rbrack_q X Y X^2 - Y X^3.
\end{align*}
\begin{lemma}
\label{def:qOns}   {\rm (See \cite[Theorem~10.1]{TD00}.)}
 Referring to the TD system $\Phi$,
\begin{align}
\label{eq:DG1}
&
\lbrack A, \lbrack A, \lbrack A, A^*\rbrack_q \rbrack_{q^{-1}}\rbrack = (q^2-q^{-2})^2 \lbrack A^*, A\rbrack,
\\
\label{eq:DG2}
& 
\lbrack A^*, \lbrack A^*, \lbrack A^*, A \rbrack_q \rbrack_{q^{-1}} \rbrack = (q^2-q^{-2})^2 \lbrack A, A^*\rbrack.
\end{align}
\end{lemma}
\noindent 
The relations {\rm (\ref{eq:DG1})}, {\rm (\ref{eq:DG2})} are called the {\it $q$-Dolan/Grady relations}.
In the following results we explore their meaning.

\begin{lemma} \label{lem:four2}
For $X \in {\rm End}(V)$ the following are equivalent:
\begin{enumerate}
\item[\rm (i)] 
$\lbrack A, \lbrack A, \lbrack A,X \rbrack_q \rbrack_{q^{-1}} \rbrack = (q^2-q^{-2})^2 \lbrack X, A \rbrack$;
\item[\rm (ii)] $A$ commutes with
\begin{align}
X + \frac{\lbrack A, \lbrack A, X \rbrack_q \rbrack_{q^{-1}}}{(q^2-q^{-2})^2}.
\label{eq:Acom}
\end{align}
\end{enumerate}
\end{lemma} 
\begin{proof} By the definition of the commutator map.
 \end{proof}

\begin{proposition}
\label{lem:four}
Let  $X$ denote an element of ${\rm End}(V)$ that acts on the eigenspaces of $A$ in a tridiagonal fashion.
Then $X$ satisfies the equivalent conditions {\rm (i), (ii)} in Lemma
\ref{lem:four2}.
\end{proposition}
\begin{proof} 
We show that $X$ satisfies Lemma \ref{lem:four2}(ii). Let $\Delta$ denote the expression in
(\ref{eq:Acom}). To show that $A$ commutes with $\Delta$, by Lemma \ref{lem:One}
 it suffices to show that $E_i \Delta E_j=0$ if $i\not=j$ $(0 \leq i,j\leq d)$.
  Let $i,j$ be given with $i \not=j$.
We have $E_i \Delta E_j = E_i X E_j c_{ij}$ where
 \begin{align*}
 c_{ij} = 1 + \frac{ q \theta_i - q^{-1} \theta_j}{q^2-q^{-2}} \,\frac{q^{-1} \theta_i - q \theta_j}{q^2-q^{-2}}.
 \end{align*}
 If $\vert i-j\vert >1$ then $E_i X E_j=0$.
 If $\vert i-j\vert = 1$ then $c_{ij}=0$ by 
 (\ref{eq:2factors}).
 In any case $E_i \Delta E_j=0$. The result follows.
\end{proof}


\noindent Later in the paper, we will encounter  pairs of elements
in ${\rm End}(V)$ that are related to each other in the following way.

\begin{proposition}
\label{lem:xyWf}
For $X, Y \in {\rm End}(V)$ the following are equivalent:
\begin{enumerate}
\item[\rm (i)] $X$ acts on the eigenspaces of $A$ in a tridiagonal fashion, and
 $A$ commutes with 
\begin{align}
Y+ \frac{ q X A- q^{-1} A X}{q^2-q^{-2}}.
\label{eq:XY}
\end{align}
\item[\rm (ii)] $Y$ acts on the eigenspaces of $A$ in a tridiagonal fashion, and
 $A$ commutes with 
\begin{align}
X+ \frac{q A Y - q^{-1} Y A}{q^2-q^{-2}}.
\label{eq:YX}
\end{align}
\end{enumerate}
\end{proposition}
\begin{proof} 
Let $C$ (resp. $D$) denote the expression in
(\ref{eq:XY}) (resp. (\ref{eq:YX})).
Note that
\begin{align}
X + \frac{ \lbrack A, \lbrack A, X \rbrack_q \rbrack_{q^{-1}}}{(q^2-q^{-2})^2} &= D - \frac{\lbrack A, C\rbrack_{q}}{q^2-q^{-2}}, \label{eq:CD1}
\\
Y + \frac{ \lbrack A, \lbrack A, Y \rbrack_q \rbrack_{q^{-1}}}{(q^2-q^{-2})^2} &= C - \frac{\lbrack D, A\rbrack_{q}}{q^2-q^{-2}}. \label{eq:CD2}
\end{align}
\noindent  ${{\rm (i)} \Rightarrow {\rm (ii)}}$
From the form of  $C$, we see that $Y$ acts on the eigenspaces of $A$ in
a tridiagonal fashion. Next we show that $A$ commutes with $D$. 
By assumption,  $X$ acts on the eigenspaces of $A$ in a tridiagonal fashion.
By this and Proposition
\ref{lem:four}, $A$ commutes with the expression on the left in
(\ref{eq:CD1}). By assumption $A$ commutes with $C$, so $A$ commutes with $\lbrack A, C\rbrack_q$. By these comments and
(\ref{eq:CD1}), $A$ commutes with $D$.
\\
\noindent ${{\rm (ii)} \Rightarrow {\rm (i)}}$ From the form of  $D$, we see that $X$ acts on the eigenspaces of $A$ in
a tridiagonal fashion. Next we show that $A$ commutes with $C$. 
By assumption, $Y$ acts on the eigenspaces of $A$ in a tridiagonal fashion.
By this and Proposition
\ref{lem:four}, $A$ commutes with the expression on the left in
(\ref{eq:CD2}). By assumption $A$ commutes with $D$, so $A$ commutes with $\lbrack D, A\rbrack_q$. By these comments and
(\ref{eq:CD2}), $A$ commutes with $C$.
\end{proof}

\section{The map $X \mapsto X^\vee$}
We continue to discuss the TD system $\Phi=(A;\lbrace E_i\rbrace_{i=0}^d;A^*; \lbrace E^*_i\rbrace_{i=0}^d)$ on $V$ that has $q$-Racah type.
 In this section, we introduce a map that will make it easier to discuss the elements in ${\rm End}(V)$ that commute with $A$.

\begin{definition}\label{def:vee}
\rm
For $X \in {\rm End}(V)$ define
\begin{align}
X^\vee = \sum_{i=0}^d E_i X E_i.
\label{eq:xv}
\end{align} 
\end{definition}

\noindent Note that the map ${\rm End}(V) \to {\rm End}(V)$, $X \to X^\vee$ is $\mathbb F$-linear. Note also that
$A$ commutes with $X^\vee $ for all $X \in {\rm End}(V)$.  Next comes a stronger statement.

\begin{lemma}\label{lem:Axv} For $X \in {\rm End}(V)$ the following are equivalent:
\begin{enumerate}
\item[\rm(i)]  $A$ commutes with $X$;
\item[\rm (ii)] $X = X^\vee$.
\end{enumerate}
\end{lemma}
\begin{proof} By Lemma \ref{lem:One}.
\end{proof}

\begin{lemma}
\label{lem:veeLinear}
For $X \in {\rm End}(V)$ the following {\rm (i)--(v)} hold:
\begin{enumerate}
\item[\rm (i)] $(AX)^\vee = A X^\vee$;
\item[\rm (ii)]  $(XA)^\vee= AX^\vee$;
\item[\rm (iii)] $\lbrack A, X \rbrack^\vee = 0$;
\item[\rm (iv)] $(\lbrack A, X \rbrack_q)^\vee = (q-q^{-1}) A X^\vee$;
\item[\rm (v)] $(\lbrack A, X \rbrack_{q^{-1}})^\vee = - (q-q^{-1}) A  X^\vee$.
 \end{enumerate}
 \end{lemma}
\begin{proof} (i), (ii) For  the given equation each side is equal to $\sum_{i=0}^d \theta_i E_i X E_i$.
\\
\noindent (iii)--(v) By (i), (ii) above.
\end{proof}

\begin{lemma} \label{lem:refxv} Let $X$ denote an element of  ${\rm End}(V)$ that acts on the eigenspaces of $A$ in a tridiagonal fashion.
Then
\begin{align}
\label{eq:aax}
X + \frac{ \lbrack A, \lbrack A, X \rbrack_q \rbrack_{q^{-1}} }{(q^2-q^{-2})^2} = \biggl( I - \frac{A^2}{(q+q^{-1})^2}\biggr) X^\vee.
\end{align}
\end{lemma}
\begin{proof}  Let $\Delta$ denote the expression on the left in
(\ref{eq:aax}). The element $A$ commutes with $\Delta $ by Proposition \ref{lem:four}, so $\Delta=\Delta^\vee$ by Lemma
\ref{lem:Axv}.
By Lemma \ref{lem:veeLinear},  $\Delta^\vee$ is equal to the expression on the right in
(\ref{eq:aax}). The result follows.
\end{proof}

\section{The element $W$}
 We continue to discuss the TD system $\Phi=(A;\lbrace E_i\rbrace_{i=0}^d;A^*; \lbrace E^*_i\rbrace_{i=0}^d)$ on $V$ that has $q$-Racah type.
 In this section we introduce a certain invertible  $W \in {\rm End}(V)$, and discuss how $W$ is related to the map $X \to X^\vee$ from Section 4.
 \begin{definition}
 \label{def:ti}
 \rm
 Define 
 \begin{align*}
 t_i = (-1)^i a^i q^{i(d-i)} \qquad \qquad (0 \leq i \leq d).
 \end{align*}
 Note that $t_0=1$ and $t_d = (-1)^d a^d$. 
 \end{definition}
 
 
 \begin{note}\rm We caution the reader that the scalar $t_i$ in 
 Definition \ref{def:ti} is a square root of the scalar $t_i$ in
 \cite[Lemma~3.15]{diagram}.
 \end{note}
 
 \begin{lemma} \label{lem:tdiv}
 We have $t_i\not=0$ for $0 \leq i \leq d$. Moreover
 \begin{align*}
 t_i /t_{i-1} = -a q^{d-2i+1} \qquad \qquad (1 \leq i \leq d).
 \end{align*}
 \end{lemma}
 \begin{proof} By Definition \ref{def:ti}.
 \end{proof}
 
 \begin{lemma} 
 \label{lem:thetat}
 For $0 \leq i,j\leq d$ such that $\vert i - j\vert  = 1$,
 \begin{align}
 \frac{t_j}{t_i}+ \frac{q \theta_i - q^{-1} \theta_j}{q^2-q^{-2}}  = 0.
 \label{eq:thetat}
 \end{align}
 \end{lemma}
 \begin{proof} Use Lemmas
 \ref{lem:onefactor}, \ref{lem:tdiv}.
 \end{proof}
 
\begin{definition}\label{def:W}
\rm
Define
\begin{align*}
W = \sum_{i=0}^d t_i E_i,
\end{align*}
where the scalars $\lbrace t_i \rbrace_{i=0}^d$ are from Definition
\ref{def:ti}.
\end{definition}
\begin{lemma} \label{lem:wi}
The element $W$ is invertible. Moreover
\begin{align*}
W^{-1} = \sum_{i=0}^d t^{-1}_i E_i.
\end{align*}
\end{lemma}
\begin{proof}  By Lemma
\ref{lem:tdiv} and Definition \ref{def:W}.
\end{proof}


\begin{note}\rm We acknowledge that the element $W$  appeared earlier in the context of 
 spin models \cite[Definition~14.2]{spinModelNT};
spin Leonard pairs \cite[Theorem~1.18]{spinLP}, 
\cite[Lemma~6.16]{spinModelNT}; 
and 
Leonard triples
\cite[Lemma~2.10]{curtMod},
 \cite[Definition~16.3]{totBip},
\cite[Definition~8.1]{terPseudo}.
In the context of TD pairs, we have $W^2=H$ where $H$ is from \cite[Section~3]{diagram}.
\end{note}
\begin{lemma} \label{lem:WA}
We have $W^{\pm 1} \in \mathcal D $.
\end{lemma}
\begin{proof} By
(\ref{eq:defEi}) and
Definition \ref{def:W}.
\end{proof}

\noindent By Lemma \ref{lem:WA}
we see that $W^{\pm 1}$ 
 are polynomials in $A$. These polynomials
will be made explicit in Section 6.

\begin{lemma} \label{lem:Wvee}
For $X \in {\rm End}(V)$,
\begin{align*}
(W^{-1} X W )^\vee = X^\vee =  (W X W^{-1})^\vee.
\end{align*}
\end{lemma}
\begin{proof} By Lemma
\ref{lem:veeLinear}(i),(ii)
and since $W^{\pm 1} \in \mathcal D$.
\end{proof}

\begin{proposition}
\label{lem:xvW}
Let $X$ denote an element of ${\rm End}(V)$ that acts on the eigenspaces of $A$ in a tridiagonal fashion. Then both
\begin{align}
\label{eq:wxv}
W^{-1} X W + \frac{ q A X - q^{-1} X A}{q^2-q^{-2}}   = \biggl(I + \frac{A}{q+q^{-1}}\biggr) X^\vee,
\\
W X W^{-1} + \frac{ q X A  - q^{-1} A X }{q^2-q^{-2}}  = 
\biggl(I + \frac{A}{q+q^{-1}}\biggr) X^\vee.
\label{eq:wixv}
\end{align}
\end{proposition}
\begin{proof} We first verify
(\ref{eq:wxv}).  Let $\Delta$ denote the expression on the left in
(\ref{eq:wxv}).
We first show that $A$ commutes with $\Delta$. By Lemma
\ref{lem:One} it suffices to show that $E_i \Delta E_j = 0$ if $i \not=j$ $(0 \leq i,j\leq d)$.
Let $i,j$ be given with $i \not= j$. Using the expression on the left in
(\ref{eq:wxv}) we have $E_i \Delta E_j = E_i X E_j c_{ij}$, where
\begin{align*}
c_{ij} = \frac{t_j}{t_i} + \frac{ q \theta_i - q^{-1} \theta_j}{q^2-q^{-2}}.
\end{align*}
If $\vert i-j\vert > 1 $ then $E_i X E_j = 0$.
If $\vert i - j \vert = 1$ then $c_{ij}=0$ by 
(\ref{eq:thetat}). 
In any case $E_i \Delta E_j=0$. Therefore  $A$ commutes with $\Delta$,
so $\Delta= \Delta^\vee$ by Lemma
\ref{lem:Axv}.
Using Lemmas
\ref{lem:veeLinear}(iv), 
\ref{lem:Wvee}
we see that $\Delta^\vee$ is equal to the expression on the right in
(\ref{eq:wxv}).
We have verified 
(\ref{eq:wxv}). One similarly verifies
(\ref{eq:wixv}).
\end{proof}

\begin{corollary}\label{cor:wixw}
Let $X$ denote an element of ${\rm End}(V)$ that  acts on the eigenspaces of $A$ in a tridiagonal fashion. Then
\begin{align}
W X W^{-1} - W^{-1} X W = \frac{\lbrack A, X\rbrack }{q-q^{-1}}.
\label{eq:WXW}
\end{align}
\end{corollary}
\begin{proof} Subtract
(\ref{eq:wxv}) from
(\ref{eq:wixv}).
\end{proof}

\noindent The following is a variation on \cite[Corollary~2.3]{terLusztig}.
\begin{corollary} \label{cor:W2XW2} 
Let $X$ denote an element of ${\rm End}(V)$ that acts on the eigenspaces of $A$ in a tridiagonal fashion.
Then
\begin{align}
W^{-2} X W^2 &=  X + \frac{\lbrack A, \lbrack A, X\rbrack_q \rbrack}{(q-q^{-1})(q^2-q^{-2})},
\label{eq:WWXWW1}
\\
W^{2} X W^{-2} &= X + \frac{\lbrack A, \lbrack A, X\rbrack_{q^{-1}} \rbrack}{(q-q^{-1})(q^2-q^{-2})}.
\label{eq:WWXWW2}
\end{align}
\end{corollary}
\begin{proof} We first obtain (\ref{eq:WWXWW1}). Recall that $A$ commutes with $W$. In (\ref{eq:WXW}), multiply each side on the left
by $W^{-1}$ and the right by $W$. This yields
\begin{align}
X- W^{-2} X W^2 = \frac{\lbrack A, W^{-1} X W \rbrack}{ q-q^{-1}}.
\label{eq:WWXstep1}
\end{align}
By 
Proposition \ref{lem:xvW},
\begin{align}
\label{eq:WWXstep2}
\lbrack A, W^{-1} X W \rbrack+
 \frac{\lbrack A, \lbrack A, X \rbrack_q \rbrack}{q^2-q^{-2}}=0.
\end{align}
Combining (\ref{eq:WWXstep1}), (\ref{eq:WWXstep2}) we obtain (\ref{eq:WWXWW1}). We similarly obtain
(\ref{eq:WWXWW2}).
\end{proof}

\begin{proposition}
\label{lem:xyW} Let $X, Y$ denote elements of ${\rm End}(V)$ that satisfy
the equivalent conditions {\rm (i), (ii)} in Proposition
\ref{lem:xyWf}.  Then
\begin{align*}
W X W^{-1}- Y  = X- W^{-1} Y W,
\end{align*}
and this common value commutes with $A$.
\end{proposition}
\begin{proof}
By (\ref{eq:wixv}), $A$ commutes with
\begin{align}
W X W^{-1} + \frac{q X A  - q^{-1} A X }{q^2-q^{-2}}.
\label{eq:AxW}
\end{align}
Combining (\ref{eq:XY}), (\ref{eq:AxW}) we see that $A$ commutes with $W X W^{-1}-Y$. We mentioned earlier that $W$ is
a polynomial in $A$. So $W$ commutes with $W X W^{-1}-Y$. Consequently
\begin{align*}
W X W^{-1}-Y = W^{-1}(W X W^{-1}-Y)W = X - W^{-1} Y W.
\end{align*}
\end{proof}

\noindent We have a comment.
\begin{lemma} \label{lem:WDown} We have
\begin{enumerate}
\item[\rm (i)]  $t^\Downarrow_i = t_{d-i}/t_d$ for $0 \leq i \leq d$;
\item[\rm (ii)]  $W^\Downarrow = t^{-1}_d W$.
\end{enumerate}
\end{lemma} 
\begin{proof} (i) Each side is equal to $(-1)^ia^{-i} q^{i(d-i)}$.
\\
\noindent (ii) By Definition \ref{def:W}
and the construction,
\begin{align*}
W^\Downarrow = \sum_{i=0}^d t^\Downarrow_i E_{d-i} = t^{-1}_d \sum_{i=0}^d t_{d-i} E_{d-i} = t^{-1}_d W.
\end{align*}
\end{proof}

\section{Some identities involving $W^{\pm 1}$}
We continue to discuss the TD system $\Phi=(A;\lbrace E_i\rbrace_{i=0}^d;A^*; \lbrace E^*_i\rbrace_{i=0}^d)$ on $V$ that has $q$-Racah type.
Recall the element $W$ from Definition
\ref{def:W}.
 In this section we obtain some identities involving $W^{\pm 1}$. Using these identities we express $W^{\pm 1}$ as a polynomial in $A$.
\medskip

\noindent We recall some notation. For $c,z\in \mathbb F$ define
 \begin{align*}
 (c;z)_n = (1-c)(1-cz) \cdots (1-cz^{n-1})
 \qquad \qquad (n \in \mathbb N).
 \end{align*}
 We will be discussing basic hypergeometric series, using the notation of
 \cite{gr,koekoek}.

 \begin{lemma}
 \label{lem:cvsum}
 For $0 \leq r\leq s \leq d$,
 \begin{align}
 \label{eq:cv}
 \frac{t_s}{t_r} &= \sum_{i=0}^{s-r} \frac{(-1)^i  q^{i^2} (\theta_s - \theta_r)(\theta_s - \theta_{r+1} )\cdots (\theta_s - \theta_{r+i-1})}{(q^2;q^2)_i(a^{-1} q^{2r+1-d};q^2)_i},
 \\
 \frac{t_r}{t_s} &= \sum_{i=0}^{s-r} \frac{ (-1)^i  q^{-i^2} (\theta_s - \theta_r)(\theta_s - \theta_{r+1} )\cdots (\theta_s - \theta_{r+i-1})}{(q^{-2};q^{-2})_i(a q^{d-2r-1};q^{-2})_i}.
\label{eq:cvinv}
 \end{align}
 \end{lemma}
 \begin{proof} To verify
 (\ref{eq:cv}), 
 evaluate the left-hand side using 
 Definition \ref{def:ti}
 and the right-hand side using 
 (\ref{eq:eigform}).
 The result becomes  a special case of the basic
  Chu/Vandermonde summation formula \cite[p.~354]{gr}:
  \begin{align*}
  (-1)^{s-r} a^{s-r} q^{(s-r)(d-r-s)} = 
 {}_2\phi_1 \biggl(
 \genfrac{}{}{0pt}{}
 {q^{2s-2r}, a^2 q^{2d-2r-2s}}
  {a q^{d-2r-1} }
  \,\bigg\vert \, q^{-2};  q^{-2} \biggr).
\end{align*}
We have verified 
(\ref{eq:cv}). To obtain (\ref{eq:cvinv}) from (\ref{eq:cv}), replace $q\mapsto q^{-1}$ and $a\mapsto a^{-1}$.
 \end{proof}
 
 \begin{proposition} 
 \label{prop:cvP}
 For $0 \leq r \leq d$ the following holds on $E_r V + E_{r+1}V + \cdots + E_d V$:
 \begin{align}
 \label{eq:wExpand}
 W &= t_r \sum_{i=0}^{d-r} \frac{ (-1)^i  q^{i^2}(A-\theta_r I )( A-\theta_{r+1}I ) \cdots (A-\theta_{r+i-1}I)}{(q^2;q^2)_i(a^{-1} q^{2r+1-d};q^2)_i},
 \\
  \label{eq:wiExpand}
 W^{-1} &= t^{-1}_r \sum_{i=0}^{d-r} \frac{ (-1)^i  q^{-i^2}(A-\theta_r I )( A-\theta_{r+1}I ) \cdots (A-\theta_{r+i-1}I)}{(q^{-2};q^{-2})_i(aq^{d-2r-1};q^{-2})_i}.
 \end{align}
 \end{proposition}
 \begin{proof} To verify (\ref{eq:wExpand}), use
 Definition \ref{def:W} 
 and
 (\ref{eq:cv})
 to see that 
 for $r \leq s \leq d$ the $E_sV$-eigenvalue for either side of
 (\ref{eq:wExpand}) is equal to $t_s$. We have verified 
 (\ref{eq:wExpand}).
 To verify (\ref{eq:wiExpand}), use
 Lemma \ref{lem:wi} and
  (\ref{eq:cvinv})
 to see that 
 for $r \leq s \leq d$ the $E_sV$-eigenvalue for either side of
 (\ref{eq:wiExpand}) is equal to $t^{-1}_s$. We have verified 
 (\ref{eq:wiExpand}).
 \end{proof}
 
 \begin{proposition} 
 \label{cor:cvP}
 The following holds on $V$:
 \begin{align*}
 W&=  \sum_{i=0}^{d} \frac{ (-1)^i q^{i^2}(A-\theta_0 I )( A-\theta_{1}I ) \cdots (A-\theta_{i-1}I)}{(q^2;q^2)_i(a^{-1} q^{1-d}; q^2)_i},
 \\
W^{-1}&=  \sum_{i=0}^{d} \frac{ (-1)^i q^{-i^2}(A-\theta_0 I )( A-\theta_{1}I ) \cdots (A-\theta_{i-1}I)}{(q^{-2};q^{-2})_i(a q^{d-1}; q^{-2})_i}.
\end{align*}
  \end{proposition}
 \begin{proof} Set $r=0$ in Proposition  \ref{prop:cvP}.
 \end{proof}

 \noindent We mention a variation on Lemma 
 \ref{lem:cvsum}
 and 
 Propositions  \ref{prop:cvP}, \ref{cor:cvP}.

 \begin{lemma}
 \label{lem:cvsumVar}
 For $0 \leq r\leq s \leq d$,
 \begin{align}
 \label{eq:cvVar}
 \frac{t_r}{t_s} &= \sum_{i=0}^{s-r} \frac{(-1)^i  q^{i^2} (\theta_r - \theta_s)(\theta_r - \theta_{s-1} )\cdots (\theta_r - \theta_{s-i+1})}{(q^2;q^2)_i(a q^{d-2s+1};q^2)_i}
 \\
 \frac{t_s}{t_r} &= \sum_{i=0}^{s-r} \frac{(-1)^i  q^{-i^2} (\theta_r - \theta_s)(\theta_r - \theta_{s-1} )\cdots (\theta_r - \theta_{s-i+1})}{(q^{-2};q^{-2})_i(a^{-1} q^{2s-d-1};q^{-2})_i}
\label{eq:cvinvVar}
 \end{align}
 \end{lemma}
 \begin{proof} 
 To verify
 (\ref{eq:cvVar}), 
 evaluate the left-hand side using 
 Definition \ref{def:ti}
 and the right-hand side using 
 (\ref{eq:eigform}).
  The result becomes  a special case of the basic
  Chu/Vandermonde summation formula \cite[p.~354]{gr}:

  \begin{align*}
  (-1)^{s-r}a^{r-s}q^{(r-s)(d-r-s)} = 
 {}_2\phi_1 \biggl(
 \genfrac{}{}{0pt}{}
 {q^{2s-2r}, a^{-2} q^{2r+2s-2d}}
  {a^{-1}q^{2s-d-1 }}
  \,\bigg\vert \, q^{-2};  q^{-2} \biggr).
\end{align*}
We have verified 
(\ref{eq:cvVar}). To obtain (\ref{eq:cvinvVar}) from (\ref{eq:cvVar}), replace $q\mapsto q^{-1}$ and $a\mapsto a^{-1}$.
 \end{proof}

 \begin{proposition} 
 \label{prop:cvPVar}
 For $0 \leq s \leq d$ the following holds on $E_0V + E_1V + \cdots +  E_sV$:
 \begin{align*}
 W &= t_s \sum_{i=0}^{s} \frac{ (-1)^i  q^{i^2}(A-\theta_s I )( A-\theta_{s-1}I ) \cdots (A-\theta_{s-i+1}I)}{(q^2;q^2)_i(a q^{d-2s+1};q^2)_i},
 \\
 W^{-1} &= t^{-1}_s \sum_{i=0}^{s} \frac{ (-1)^i  q^{-i^2}(A-\theta_s I )( A-\theta_{s-1}I ) \cdots (A-\theta_{s-i+1}I)}{(q^{-2};q^{-2})_i(a^{-1} q^{2s-d-1};q^{-2})_i}.
 \end{align*}
 \end{proposition}
 \begin{proof} Similar to the proof of
 Proposition \ref{prop:cvP}.
 \end{proof}

 \begin{proposition} 
 \label{cor:cvPVar}
 The following holds on $V$:
 \begin{align*}
 W &=t_d \sum_{i=0}^{d} \frac{ (-1)^i  q^{i^2}(A-\theta_d I )( A-\theta_{d-1}I ) \cdots (A-\theta_{d-i+1}I)}{(q^2;q^2)_i(a q^{1-d};q^2)_i},
 \\
 W^{-1} &= t^{-1}_d \sum_{i=0}^{d} \frac{ (-1)^i  q^{-i^2}(A-\theta_d I )( A-\theta_{d-1}I ) \cdots (A-\theta_{d-i+1}I)}{(q^{-2};q^{-2})_i(a^{-1} q^{d-1};q^{-2})_i}.
 \end{align*}
 \end{proposition}
 \begin{proof} Set $s=d$ in 
 Proposition \ref{prop:cvPVar}.
 \end{proof}


\section{The elements $K$, $B$}
We continue to discuss the TD system $\Phi=(A;\lbrace E_i\rbrace_{i=0}^d;A^*; \lbrace E^*_i\rbrace_{i=0}^d)$ on $V$ that has $q$-Racah type.
In this section we recall the elements $ K, B \in {\rm End}(V)$ and discuss their basic properties.

\begin{definition}\label{def:decomp}
\rm By a {\it decomposition of $V$} we mean
a sequence $\lbrace V_i\rbrace_{i=0}^d$ of nonzero subspaces whose direct sum is $V$.
\end{definition}

\noindent For $0 \leq i \leq d$ define
\begin{align*} 
U_i = (E^*_0V + E^*_1V + \cdots + E^*_iV)\cap 
(E_iV + E_{i+1}V + \cdots + E_dV).
\end{align*}
By \cite[Theorem~4.6]{TD00}  the sequence $\lbrace U_i\rbrace_{i=0}^d$ is a decomposition of $V$. We call 
$\lbrace U_i\rbrace_{i=0}^d$ the {\it first split decomposition} of $V$. By \cite[Theorem~4.6]{TD00} the
following hold for $0 \leq i \leq d$:
\begin{align}
E^*_0V + E^*_1V+\cdots + E^*_iV &= U_0 + U_1 + \cdots + U_i,
\label{eq:s1}
\\
E_iV + E_{i+1}V + \cdots + E_dV &= U_i + U_{i+1} + \cdots + U_d.
\label{eq:s2}
\end{align}
Also by \cite[Theorem~4.6]{TD00},
\begin{align}
&(A-\theta_i I )U_i \subseteq U_{i+1} \qquad (0 \leq i \leq d-1), \quad (A-\theta_d I )U_d=0,
\label{eq:AUW}
\\
&(A^*-\theta^*_i I )U_i \subseteq U_{i-1} \qquad (1 \leq i \leq d), \quad (A^*-\theta^*_0 I )U_0=0.
\label{eq:AsUW}
\end{align}
\noindent By these remarks and our comments above Definition
\ref{def}, we obtain the following.
For $0 \leq i \leq d$,
\begin{align*} 
U^\Downarrow_i = (E^*_0V + E^*_1V + \cdots + E^*_iV)\cap 
(E_0V + E_1V + \cdots + E_{d-i}V).
\end{align*}
The sequence $\lbrace U^\Downarrow_i\rbrace_{i=0}^d$ is a decomposition of $V$. We call
$\lbrace U^\Downarrow_i\rbrace_{i=0}^d$   the {\it second split decomposition} of $V$.
For $0 \leq i \leq d$,
\begin{align}
E^*_0V + E^*_1V+\cdots + E^*_iV &= U^\Downarrow_0 + U^\Downarrow_1 + \cdots + U^\Downarrow_i,
\label{eq:s3}
\\
E_0V + E_1V + \cdots + E_{d-i}V &= U^\Downarrow_i + U^\Downarrow_{i+1} + \cdots + U^\Downarrow_d.
\label{eq:s4}
\end{align}
We have
\begin{align*}
&(A-\theta_{d-i} I )U^\Downarrow_i \subseteq U^\Downarrow_{i+1} \qquad (0 \leq i \leq d-1), \quad (A-\theta_0I )U^\Downarrow_d=0,
\\
&(A^*-\theta^*_{i} I )U^\Downarrow_i \subseteq U^\Downarrow_{i-1} \qquad (1 \leq i \leq d), \quad (A^*-\theta^*_0I )U^\Downarrow_0=0.
\end{align*}

\begin{definition}
\label{def:K}\rm (See \cite[Section~1.1]{augmented}.)
Define $K\in {\rm End}(V)$ such that for $0 \leq i \leq d$, $U_i$  is an eigenspace for $K$ with
eigenvalue $q^{d-2i}$. Define $B=K^\Downarrow$.
So for $0 \leq i \leq d$, $U^\Downarrow_i$ is an eigenspace for $B$ with
eigenvalue $q^{d-2i}$. 
\end{definition}
\noindent  By construction $K$, $B$ are invertible. The elements $A$, $K$, $B$ are related as follows.
By \cite[Section~1.1]{augmented},
\begin{align}
&\frac{q K A - q^{-1} A K}{q-q^{-1}} = a K^2+ a^{-1} I,\qquad \qquad 
\frac{q B A - q^{-1} A B}{q-q^{-1}} =a^{-1} B^2+ a I.
\label{eq:kA}
\end{align}
\noindent By \cite[Theorem~9.9]{bockting},
\begin{align}
& aK^2 - \frac{a^{-1} q - a q^{-1}}{q-q^{-1}} K B - \frac{a q - a^{-1} q^{-1}}{q-q^{-1}} BK + a^{-1} B^2 = 0.
\label{eq:kb}
\end{align}
The  equations 
(\ref{eq:kA}),
(\ref{eq:kb})
can be reformulated as follows. By
\cite[Lemma~12.12]{bocktingTer},
\begin{align}
&\frac{q A K^{-1}  - q^{-1}K^{-1}A }{q-q^{-1}} = a^{-1} K^{-2}+ aI,\qquad \quad 
\frac{q A B^{-1} - q^{-1} B^{-1} A }{q-q^{-1}} =a B^{-2}+ a^{-1} I.
\label{eq:kAa}
\end{align}
\noindent By \cite[Theorem~9.10]{bockting},
\begin{align}
& a^{-1}K^{-2} - \frac{a^{-1} q - a q^{-1}}{q-q^{-1}} K^{-1} B^{-1} - \frac{a q - a^{-1} q^{-1}}{q-q^{-1}} B^{-1}K^{-1} + a B^{-2} = 0.
\label{eq:kba}
\end{align}

\noindent We now bring in $W$.
\begin{lemma} \label{lem:KBW} {\rm (See \cite[Proposition~6.1]{diagram}.)}
We have
\begin{align}
\label{eq:W2KW2}
&W^{-2} K W^2 = a^{-1} A - a^{-2} K^{-1}, \qquad \qquad
W^{-2} B W^2 = a A - a^{2} B^{-1},
\\
& W^2 K^{-1} W^{-2} = a A - a^2 K, \qquad \qquad
W^2 B^{-1} W^{-2} = a^{-1} A - a^{-2} B.
\label{eq:W2KW2a}
\end{align}
\end{lemma}

\begin{proposition} \label{cor:WKB}
We have
\begin{align}
\label{eq:wwk}
\frac{ q W^{-2}K W^2 K^{-1} - q^{-1}  K^{-1} W^{-2} KW^2} {q-q^{-1}} &= I, 
\\
\label{eq:wwkalt}
\frac{ q K W^2 K^{-1}W^{-2} - q^{-1}  W^2 K^{-1} W^{-2} K} {q-q^{-1}} &= I
\end{align}
\noindent and
\begin{align}
\frac{ q  W^{-2} B W^2 B^{-1}- q^{-1}  B^{-1} W^{-2} BW^2} {q-q^{-1}} &= I,
\label{eq:wwb}
\\
\frac{ q B W^2 B^{-1}W^{-2} - q^{-1}  W^2 B^{-1} W^{-2} B} {q-q^{-1}} &= I.
\label{eq:wwbalt}
\end{align}
\end{proposition}
\begin{proof}  To verify 
(\ref{eq:wwk}), eliminate $W^{-2} K W^2$ using
 the equation on the left in
(\ref{eq:W2KW2}), and evaluate the result using the equation on the left in (\ref{eq:kAa}).
To obtain (\ref{eq:wwkalt}), multiply each side of 
(\ref{eq:wwk}) on the left by $W^2$ and the right by $W^{-2}$.
The equations
(\ref{eq:wwb}), 
(\ref{eq:wwbalt}) 
are similary verified.
\end{proof}

\begin{proposition} 
\label{prop:wkw}
 We have
\begin{enumerate}
\item[\rm (i)] $A = a W^{-1} K W + a^{-1} W K^{-1} W^{-1}$;
\item[\rm (ii)] $A = a^{-1} W^{-1} B W + a W B^{-1} W^{-1}$.
\end{enumerate}
\end{proposition}
\begin{proof} (i) In the equation $W^{-2} K W^2 = a^{-1} A - a^{-2} K^{-1}$,
multiply each side on the left by $W$ and the right by $W^{-1}$. Evaluate the result using the
fact that $A, W$ commute.
\\
\noindent (ii) In the equation $W^{-2} B W^2 = a A - a^{2} B^{-1}$,
multiply each side on the left by $W$ and the right by $W^{-1}$. Evaluate the result using the
fact that $A, W$ commute.
\end{proof}

\begin{corollary}
\label{cor:WMW} We have
\begin{align}
\label{eq:Qpre}
W^{-1} \frac{ a K - a^{-1} B}{ a - a^{-1}} W = W \frac{ a^{-1} K^{-1} - a B^{-1}}{a^{-1}-a} W^{-1}.
\end{align}
\end{corollary}
\begin{proof} Compare the two equations in
Proposition
\ref{prop:wkw}.
\end{proof}

\section{The elements $M$, $N$, $Q$}
We continue to discuss the TD system $\Phi=(A;\lbrace E_i\rbrace_{i=0}^d;A^*; \lbrace E^*_i\rbrace_{i=0}^d)$ on $V$ that has $q$-Racah type.
Recall the maps $K, B$ from Definition \ref{def:K}. In this section we use $K,B$ to define some elements
$M, N, Q \in {\rm End}(V)$ that will play a role in our theory.
\medskip

\noindent Following 
\cite[Section~6]{bocktingQexp} and \cite[Section~7]{diagram}
we define
\begin{align}
\label{eq:MNdef}
M = \frac{a K - a^{-1} B}{a-a^{-1}},
\qquad \qquad
N = \frac{a^{-1} K^{-1} - a B^{-1} }{a^{-1} -a}.
\end{align}
By \cite[Lemma~8.1]{bocktingQexp},
 $M$ is diagonalizable with eigenvalues $\lbrace q^{d-2i} \rbrace_{i=0}^d$. The same holds for $N$ by
 Corollary \ref{cor:WMW}.  The elements $M$, $N$ 
are invertible. By construction $M^\Downarrow = M$ and $N^\Downarrow = N$. Also by construction,
\begin{align}
 K N B = M = B N K.
\label{eq:KNB}
\end{align}

\begin{lemma} \label{lem:MN}
Each of $M^{-1}$, $N^{-1}$ acts on the eigenspaces of $A$ in a tridiagonal fashion.
\end{lemma}
\begin{proof} This holds for $M^{-1}$ by 
\cite[Lemma~10.3]{bocktingQexp}. 
It holds for $N^{-1}$,  by Corollary
\ref{cor:WMW}
 and since
$A$ commutes with $W$.
\end{proof}


\begin{definition}
\label{def:Q}
\rm
By Corollary
\ref{cor:WMW}
 we have 
\begin{align}
\label{eq:Q}
W^{-1} M W = W N W^{-1};
\end{align}
this  common value will be denoted by $Q$.
\end{definition}

\begin{lemma}  The element $Q$ is diagonalizable, with eigenvalues $\lbrace q^{d-2i}\rbrace_{i=0}^d$. Moreover $Q$
is invertible.
\end{lemma}
\begin{proof} By Definition \ref{def:Q} and the comments above Lemma
\ref{lem:MN}.
\end{proof}

\begin{lemma} \label{lem:QDown}
We have $Q^\Downarrow = Q$.
\end{lemma}
\begin{proof}  By Lemma \ref{lem:WDown}(ii) and Definition \ref{def:Q}, along with the last sentence before
Lemma \ref{lem:MN}.
\end{proof}

\begin{lemma} \label{lem:MNQ}
The element $Q^{-1}$ acts on the eigenspaces of $A$ in a tridiagonal fashion.
\end{lemma}
\begin{proof}  By Lemma
\ref{lem:MN},  Definition \ref{def:Q}, and since $A$ commutes with $W$.
\end{proof}

\section{Equitable triples}
We continue to discuss the TD system $\Phi=(A;\lbrace E_i\rbrace_{i=0}^d;A^*; \lbrace E^*_i\rbrace_{i=0}^d)$ on $V$ that has $q$-Racah type.
 The two equations in
Proposition \ref{prop:wkw} each express $A$ as a sum of two terms. In this section
we describe how these terms are related to the element $Q$ from Definition
\ref{def:Q}. To facilitate this description, we will use the notion of an equitable triple.

\begin{definition} \rm (See \cite[Definition~7.3]{diagram}.)
\label{def:etr} An {\it equitable triple on $V$} is 
a $3$-tuple $X,Y,Z$ of invertible elements in ${\rm End}(V)$ such that
\begin{align*}
\frac{qXY-q^{-1} YX}{q-q^{-1} } = I,
\qquad \quad 
\frac{qYZ-q^{-1} ZY}{q-q^{-1} } = I,
\qquad \quad 
\frac{qZX-q^{-1} XZ}{q-q^{-1} } = I.
\end{align*}
\end{definition}
\noindent   Equitable triples are related to the quantum group $U_q(\mathfrak{sl}_2)$; see Section 14 below and
\cite{equit}, \cite{uawEquit},
\cite{fduq}. 

\begin{lemma}
\label{lem:4equit} {\rm (See \cite[Proposition~7.4]{diagram}.)}
Each of the following {\rm (i)--(iv)} is an equitable triple:
\begin{enumerate}
\item[\rm (i)] $a A - a^2K, M^{-1}, K$;
\item[\rm (ii)] $a^{-1} A - a^{-2}B, M^{-1}, B$;
\item[\rm (iii)] $K^{-1}, N^{-1}, a^{-1}  A - a^{-2}K^{-1}$;
\item[\rm (iv)] $B^{-1}, N^{-1}, a  A - a^2B^{-1}$.
\end{enumerate}
\end{lemma}

\begin{proposition}
\label{prop:twoET}
Each of the following {\rm (i), (ii)} is an equitable triple:
\begin{enumerate}
\item[\rm (i)] $W K^{-1} W^{-1}, Q^{-1}, W^{-1} K W$;
\item[\rm (ii)] $W B^{-1} W^{-1}, Q^{-1}, W^{-1} B W$.
\end{enumerate}
\end{proposition}
\begin{proof}(i) Define $X=a A-a^2 K$, $Y=M^{-1}$, $Z=K$. By Lemma \ref{lem:4equit}(i)
the three-tuple $X,Y,Z$  is an equitable triple.
Therefore the three-tuple $W^{-1}X W$, $W^{-1}YW$, $W^{-1}ZW$ is an equitable triple.
Using Proposition
\ref{prop:wkw}(i) we obtain $W^{-1}XW=W K^{-1} W^{-1}$. By construction $W^{-1}Y W = Q^{-1}$ and $W^{-1} ZW=W^{-1}K W$.
The result follows.
\\
\noindent (ii) Similar to the proof of (i) above, using the equitable triple from Lemma \ref{lem:4equit}(ii).
\end{proof}

\section{The double lowering map $\psi$}
We continue to discuss the TD system $\Phi=(A;\lbrace E_i\rbrace_{i=0}^d;A^*; \lbrace E^*_i\rbrace_{i=0}^d)$ on $V$ that has $q$-Racah type.
In this section we recall the double lowering map $\psi$ and discuss its basic properties.
\medskip

\noindent Recall the maps $K, B$ from Definition \ref{def:K}. By \cite[Lemma~9.7]{bockting}, each of the following is invertible:
\begin{align*}
&a I - a^{-1} B K^{-1}, \qquad \qquad a^{-1} I - a K B^{-1},
\\
& a I - a^{-1} K^{-1} B, \qquad \qquad a^{-1} I  - a B^{-1} K.
\end{align*}

\begin{lemma}\label{lem:psi4} {\rm (See \cite[Theorem~9.8]{bockting}.)}
The following coincide:
\begin{align*}
&
\frac{ I - BK^{-1}}{q(a I - a^{-1} B K^{-1})}, 
\qquad \qquad 
\frac{ I - KB^{-1}}{q(a^{-1} I - a K B^{-1})},
\\
&
\frac{ q(I-K^{-1}B)}{a I - a^{-1} K^{-1}B}, \qquad \qquad \quad \,
\frac{q(I-B^{-1}K)}{a^{-1} I - a B^{-1} K}.
\end{align*}
\end{lemma}

\begin{definition}\label{def:psi} \rm (See \cite[Definition~4.3]{bocktingQexp}.)
 Define $\psi \in {\rm End}(V)$ to be the common value
of the four expressions in Lemma
\ref{lem:psi4}.
\end{definition}

\begin{lemma} {\rm (See \cite[Lemma~5.4]{bockting}.)} Both
\begin{align*}
K \psi = q^2 \psi K, \qquad \qquad B \psi = q^2 \psi B.
\end{align*}
\end{lemma}

\begin{lemma} \label{lem:psiDown}
 {\rm (See \cite[Corollary~15.2]{bockting1}.)}  We have $\psi^\Downarrow = \psi$.
\end{lemma}

\begin{lemma}\label{lem:dd} 
{\rm (See \cite[Lemma~11.2, Corollary~15.3]{bockting1}.)}  
We have
\begin{align*} 
\psi U_i \subseteq U_{i-1}, \qquad \qquad 
\psi U^\Downarrow_i \subseteq U^\Downarrow_{i-1} \qquad \qquad (0 \leq i \leq d),
\end{align*}
where $U_{-1}=0$ and $U^\Downarrow_{-1} = 0$.
\end{lemma}
\noindent Motivated by Lemma \ref{lem:dd}, the map $\psi$ is often called the {\it double lowering map} for $\Phi$.

\begin{lemma}
\label{lem:psi3} {\rm (See \cite[Corollary~15.4]{bockting1}.)}
The element $\psi$ acts on the eigenspaces of $A$ in a tridiagonal fashion.
\end{lemma}

\begin{lemma} \label{lem:M4} 
{\rm (See \cite[Lemma~6.4]{bocktingQexp}.)}
The element $M^{-1}$ is equal to each of the following:
\begin{align*}
&K^{-1} (I - a^{-1} q \psi), \qquad \qquad (I-a^{-1} q^{-1} \psi ) K^{-1},
\\
&B^{-1} (I - a q \psi), \qquad \qquad \quad (I-a q^{-1} \psi ) B^{-1}.
\end{align*}
\end{lemma}

\begin{lemma} \label{lem:N4} The element $N^{-1}$ is equal to each of the following:
\begin{align*}
&K (I - a q^{-1} \psi), \qquad \qquad \quad (I-a q \psi ) K,
\\
&B (I - a^{-1} q^{-1} \psi), \qquad \qquad  (I-a^{-1} q \psi ) B.
\end{align*}
\end{lemma}
\begin{proof} Use (\ref{eq:KNB}) and Lemma \ref{lem:M4}.
\end{proof}

\noindent The equation (\ref{eq:lam1}) below appears in \cite[Lemma~6.8]{bocktingQexp}; we will give a short proof for the
sake of completness.
\begin{proposition} 
\label{lem:lam12}
We have
\begin{align}
\label{eq:lam1}
\psi + \frac{q  A M^{-1} - q^{-1} M^{-1} A}{q^2-q^{-2}} &= \frac{a+a^{-1}}{q+q^{-1}} I,
\\
\psi + \frac{q N^{-1} A - q^{-1} AN^{-1} }{q^2-q^{-2}}  &= \frac{a+a^{-1}}{q+q^{-1}} I.
\label{eq:lam2}
\end{align}
\end{proposition}
\begin{proof} We first obtain
(\ref{eq:lam1}). Abbreviate  $X=aA-a^2K$ and $Y=M^{-1}$. The elements $X, Y$ are the first two terms in
the equitable triple from Lemma
\ref{lem:4equit}(i). So $qXY-q^{-1} Y X = (q-q^{-1})I$. In this equation, eliminate the products
$KM^{-1}$, $M^{-1}K$ 
 using the equations
$KM^{-1} = 1 - a^{-1} q \psi$ and $M^{-1} K = 1 - a^{-1} q^{-1} \psi$ from
Lemma \ref{lem:M4}. This yields  (\ref{eq:lam1}). The equation
(\ref{eq:lam2}) is similarly obtained, using the last two terms in the equitable triple from
Lemma \ref{lem:4equit}(iii).
\end{proof}


\section{The Casimir element $\Lambda$}
We continue to discuss the TD system $\Phi=(A;\lbrace E_i\rbrace_{i=0}^d;A^*; \lbrace E^*_i\rbrace_{i=0}^d)$ on $V$ that has $q$-Racah type.
In this section we recall the Casimir element $\Lambda$, and discuss its basic properties.

\begin{lemma}
\label{lem:cas4} {\rm (See \cite[Lemmas~7.2,~8.2,~9.1]{bockting}.)}
The following coincide:
\begin{align*}
&   \psi (A - a K - a^{-1} K^{-1}) + q^{-1} K + q K^{-1},
\\
&    (A - a K - a^{-1} K^{-1})\psi + q K + q^{-1} K^{-1},
\\
&   \psi (A - a^{-1} B - a B^{-1}) + q^{-1} B + q B^{-1},
\\
&    (A - a^{-1} B - a B^{-1})\psi + q B + q^{-1} B^{-1}.
\end{align*}
\end{lemma}

\begin{definition}
\label{def:cas}
\rm Let $\Lambda$ denote the common value of the four expressions in Lemma
\ref{lem:cas4}.
\end{definition}

\begin{lemma} \label{lem:casC} 
The element 
$\Lambda$ commutes with each of
$A, W, K, B, M, N, Q, \psi$.
\end{lemma}
\begin{proof} It was shown in \cite[Lemma~7.3, 8.3, 9.1]{bockting} 
 that $\Lambda$ commutes with $A, K, B, \psi$. Now 
$\Lambda$ commutes with $W, M, N, Q$ by Lemma
\ref{lem:WA}, line (\ref{eq:MNdef}),  and Definition
\ref{def:Q}.
\end{proof}

\noindent Motivated by Lemma \ref{lem:casC}, we call $\Lambda$ the {\it Casimir element} for $\Phi$.

\begin{lemma} \label{lem:LamDown}
We have $\Lambda^\Downarrow = \Lambda$.
\end{lemma}
\begin{proof} By Lemmas
 \ref{lem:psiDown},
\ref{lem:cas4} and $K^\Downarrow = B$.
\end{proof}

\begin{lemma}
\label{lem:Apsi}
We have
\begin{enumerate}
\item[\rm (i)] $A \psi = \Lambda - q N^{-1} - q^{-1} M^{-1}$;
\item[\rm (ii)] $ \psi A  = \Lambda - q^{-1} N^{-1} - q M^{-1}$.
\end{enumerate}
\end{lemma}
\begin{proof} (i) By Definition \ref{def:cas}
we have  $\Lambda = (A-aK-a^{-1}K^{-1})\psi + q K + q^{-1} K^{-1}$; evaluate this equation using
$M^{-1} = K^{-1}(1-a^{-1} q \psi)$ and
$N^{-1} = K(1-a q^{-1} \psi)$.
\\
\noindent (ii) Similar to the proof of (i) above.
\end{proof}

\begin{proposition}\label{lem:MLam}
We have
\begin{align}
M^{-1} + \frac{q \psi A - q^{-1} A \psi}{q^2-q^{-2}}  &= \frac{\Lambda}{q+q^{-1}},
\label{eq:psip1}
\\
N^{-1}+ \frac{q  A\psi - q^{-1}  \psi A}{q^2-q^{-2}} &= \frac{\Lambda}{q+q^{-1}}.
\label{eq:psip2}
\end{align}
\end{proposition}
\begin{proof} Use Lemma
\ref{lem:Apsi}.
\end{proof}

\section{The element $\psi- Q^{-1}$}
We continue to discuss the TD system $\Phi=(A;\lbrace E_i\rbrace_{i=0}^d;A^*; \lbrace E^*_i\rbrace_{i=0}^d)$ on $V$ that has $q$-Racah type.
In this section we investigate the element $\psi - Q^{-1}$, where
$Q$ is from
Definition \ref{def:Q}
and
 $\psi$ is from
Definition   \ref{def:psi}.

\begin{lemma} \label{cor:mp}
Each of the following pairs satisfy the equivalent conditions {\rm (i), (ii)} in
Proposition
\ref{lem:xyWf}:
 \begin{enumerate}
 \item[\rm (i)]  $\psi$, $M^{-1}$;
 \item[\rm (ii)] $N^{-1}$, $\psi$.
 \end{enumerate}
\end{lemma}
\begin{proof} By 
Lemmas
\ref{lem:MN}, \ref{lem:psi3}, \ref{lem:casC} and
Propositions 
\ref{lem:lam12},
\ref{lem:MLam}.
\end{proof}

\begin{proposition}\label{prop:Acom}
The element $A$ commutes with $\psi-Q^{-1}$.
Moreover,
\begin{align}
W \psi W^{-1} + Q^{-1} = \psi + M^{-1},
\qquad \qquad
W^{-1} \psi W + Q^{-1} = \psi + N^{-1}.
\end{align}
\end{proposition}
\begin{proof}  By 
Proposition \ref{lem:xyW}, Definition \ref{def:Q}, 
and Lemma \ref{cor:mp}.
\end{proof}

\begin{proposition}\label{prop:Alam} We have
\begin{align*}
(\psi - Q^{-1})((q+q^{-1})I-A) = (a+a^{-1})I - \Lambda.
\end{align*}
\end{proposition}
\begin{proof} The element $A$ commutes with $I$ and $\Lambda$. So by Lemma
\ref{lem:Axv}, $I=I^\vee$ and $\Lambda = \Lambda^\vee$. By Proposition
\ref{prop:Acom}, $A$ commutes with $\psi - Q^{-1}$.  So by Lemma \ref{lem:Axv},
\begin{align}
\label{eq:half}
 \psi - Q^{-1} = \psi^\vee - (Q^{-1})^\vee.
 \end{align} 
 By Lemma \ref{lem:Wvee} and
 Definition
 \ref{def:Q},
 \begin{align}
  (Q^{-1})^\vee = (M^{-1})^\vee.
  \label{eq:QQM}
  \end{align}
For the equation
(\ref{eq:lam1}), apply the map $\vee$ to each side and evaluate the result using Lemma \ref{lem:veeLinear}  along with  $I = I^\vee$; this yields
\begin{align}
\label{eq:vee1}
\psi^{\vee}+ \frac{ A ( M^{-1})^\vee}{q+q^{-1}}  = \frac{a+a^{-1}}{q+q^{-1}} I.
\end{align}
For the equation (\ref{eq:psip1}), apply the map $\vee$ to each side and evaluate the result using Lemma \ref{lem:veeLinear}
along with $\Lambda = \Lambda^\vee $; this yields
\begin{align}
\label{eq:vee2}
(M^{-1})^\vee + \frac{ A \psi^\vee}{q+q^{-1}}  = \frac{\Lambda}{q+q^{-1}}.
\end{align}
To finish the proof,
subtract (\ref{eq:vee2}) from (\ref{eq:vee1}) and evaluate the result using
(\ref{eq:half}), (\ref{eq:QQM}).
\end{proof}

\section{How   $W$, $K$ are related and how $W$, $B$  are related}
We continue to discuss the TD system $\Phi=(A;\lbrace E_i\rbrace_{i=0}^d;A^*; \lbrace E^*_i\rbrace_{i=0}^d)$ on $V$ that has $q$-Racah type.
In Proposition \ref{cor:WKB}
we showed how $W^2$, $K$ are related and how $W^2$, $B$ are related.
 In the present section we show how $W, K$
are related and how $W, B$ are related. We will use an identity from Section 6.

\begin{proposition}\label{thm:WK} We have
\begin{align}
a W^{-1} K  W - q I &= K (a I - q W K^{-1} W^{-1}),
\label{eq:WK1}
\\
a W^{-1} K  W - q^{-1} I &= (a I - q^{-1} W K^{-1} W^{-1})K.
\label{eq:WK2}
\end{align}
\end{proposition}
\begin{proof} 
We first obtain (\ref{eq:WK1}). To this end, it is convenient to make a change of variables.
In (\ref{eq:WK1}), eliminate $W^{-1}KW$ using Proposition \ref{prop:wkw}(i), and in the result eliminate $A$ using
\begin{align}
R = A - aK - a^{-1} K^{-1}.
\label{eq:R}
\end{align} 
This yields
\begin{align}
\label{eq:PsiComGoal}
(a^{-1}I - qK)(WK^{-1}-K^{-1}W) = RW.
\end{align}
We will verify (\ref{eq:PsiComGoal}) after a few comments. Recall the first split decomposition
 $\lbrace U_i \rbrace_{i=0}^d$ of $V$.
 By Definition \ref{def:K},
  $K=q^{d-2i}I$ on $U_i$ for $0 \leq i \leq d$.
 So for $0 \leq i \leq d$ the following holds on $U_i$:
\begin{align}
a K+ a^{-1} K^{-1} = \theta_i I .
\label{eq:aKW}
\end{align}
By  (\ref{eq:R}), (\ref{eq:aKW}) we find that for
$0 \leq i \leq d$ the following holds on $U_i$:
\begin{align}
\label{eq:RAW}
R = A - \theta_i I.
\end{align}
By (\ref{eq:AUW})
 and (\ref{eq:RAW}),
\begin{align}
\label{eq:RUW}
R U_i \subseteq U_{i+1} \qquad \quad (0 \leq i \leq d-1), \qquad RU_d=0.
\end{align}
By (\ref{eq:RUW}) and the construction,
\begin{align*}
RK = q^2 KR.
\end{align*}
For $0 \leq r \leq d$ we show that 
(\ref{eq:PsiComGoal}) holds on $U_r$.
Using
(\ref{eq:RAW}), 
(\ref{eq:RUW}) 
we find that for $0 \leq i \leq d-r$ the following holds on $U_r$:
\begin{align}
R^i = (A-\theta_r I ) (A-\theta_{r+1}I) \cdots (A-\theta_{r+i-1}I).
\label{eq:RiW}
\end{align}
Also by 
(\ref{eq:RUW})  we have $R^{d-r+1}=0$ on $U_r$.
By (\ref{eq:s2}),
\begin{align*} 
E_rV + E_{r+1}V + \cdots + E_dV =
U_r + U_{r+1} + \cdots + U_d.
\end{align*}
The above subspace contains $U_r$, so 
by Proposition \ref{prop:cvP}
and (\ref{eq:RiW})
the following holds on $U_r$:
\begin{align*}
W = t_r \sum_{i=0}^{d-r} \frac{(-1)^i q^{i^2} R^i }{(q^2;q^2)_i(a^{-1} q^{2r+1-d};q^2)_i}.
\end{align*}
We may now argue that on $U_r$,
\begin{align*}
(a^{-1}I - qK)(WK^{-1} - K^{-1} W) 
&=(a^{-1} I - qK) 
 t_r \sum_{i=0}^{d-r} \frac{(-1)^i q^{i^2} (R^iK^{-1}-K^{-1} R^i) }{(q^2;q^2)_i(a^{-1} q^{2r+1-d};q^2)_i}
\\ 
&=(a^{-1} I - qK) 
 t_r \sum_{i=0}^{d-r} \frac{(-1)^i q^{i^2} R^iK^{-1}(1-q^{2i}) }{(q^2;q^2)_i(a^{-1} q^{2r+1-d};q^2)_i}
\\ 
&=(a^{-1} I - qK) 
 t_r \sum_{i=1}^{d-r} \frac{(-1)^i q^{i^2} R^iK^{-1}(1-q^{2i}) }{(q^2;q^2)_i(a^{-1} q^{2r+1-d};q^2)_i}
\\ 
&=
 t_r \sum_{i=1}^{d-r} \frac{(-1)^i q^{i^2} R^i(a^{-1} I - q^{1-2i}K) K^{-1}(1-q^{2i}) }{(q^2;q^2)_i(a^{-1} q^{2r+1-d};q^2)_i}
\\ 
&=
-t_r \sum_{i=1}^{d-r} \frac{(-1)^{i} q^{i^2} R^iq^{1-2i}(1-a^{-1}q^{2r-d+2i-1}) (1-q^{2i})}{(q^2;q^2)_i(a^{-1} q^{2r+1-d};q^2)_i}
\\ 
&=
t_r \sum_{i=1}^{d-r} \frac{(-1)^{(i-1)} q^{(i-1)^2} R^i}{(q^2;q^2)_{i-1}(a^{-1} q^{2r+1-d};q^2)_{i-1}}
\\ 
&=
t_r \sum_{i=0}^{d-r-1} \frac{(-1)^{i} q^{i^2} R^{i+1}}{(q^2;q^2)_{i}(a^{-1} q^{2r+1-d};q^2)_{i}}
\\ 
&=
t_r \sum_{i=0}^{d-r} \frac{(-1)^{i} q^{i^2} R^{i+1}}{(q^2;q^2)_{i}(a^{-1} q^{2r+1-d};q^2)_{i}}
\\ 
&= RW.
\end{align*}
We have obtained
(\ref{eq:PsiComGoal}), and  (\ref{eq:WK1})
follows.
Next we obtain (\ref{eq:WK2}). Let $E$ denote the equation obtained by adding (\ref{eq:WK1}) to the equation in Proposition  \ref{prop:wkw}(i).
Then $W^{-1} E W$ minus the equation in Proposition \ref{prop:wkw}(i) is equal to
 (\ref{eq:WK2})  times $a^{-1} q K^{-1}$. This gives (\ref{eq:WK2}).
\end{proof}

\begin{corollary}\label{cor:combine}
We have
\begin{align}
\frac{q W^{-1} K W K^{-1} - q^{-1}  K^{-1} W^{-1} K W}{q-q^{-1} } &= I,
\label{eq:wkwk}
\\
\frac {q K W K^{-1} W^{-1}- q^{-1} W K^{-1} W^{-1} K }{q-q^{-1} } &= I.
\label{eq:wkwkalt}
\end{align}
\end{corollary}
\begin{proof} To obtain (\ref{eq:wkwkalt}), 
subtract (\ref{eq:WK1}) from 
(\ref{eq:WK2}). 
To obtain (\ref{eq:wkwk}), multiply each side of
(\ref{eq:wkwkalt}) on the left by $W^{-1}$ and the right by $W$.
\end{proof}

\begin{proposition}\label{thm:WB} We have
\begin{align}
a I  - q W^{-1} B W &= (a W B^{-1} W^{-1} - q I) B,
\label{eq:WB1}
\\
a I  - q^{-1} W^{-1} B W &= B(a W B^{-1} W^{-1} - q^{-1} I) .
\label{eq:WB2}
\end{align}
\end{proposition}
\begin{proof} Apply Proposition \ref{thm:WK} to $\Phi^{\Downarrow}$, and use $K^\Downarrow=B$ along with
Lemma \ref{lem:WDown}(ii).
\end{proof}

\begin{corollary}\label{cor:combine2}
We have
\begin{align}
\frac{q W^{-1} B W B^{-1} - q^{-1}  B^{-1} W^{-1} BW}{q-q^{-1} } &= I,
\label{eq:wbwb}
\\
\frac{q B W B^{-1}W^{-1} - q^{-1}  W B^{-1} W^{-1} B}{q-q^{-1} } &= I.
\label{eq:wbwbalt}
\end{align}
\end{corollary}
\begin{proof}  Apply Corollary \ref{cor:combine} to $\Phi^{\Downarrow}$, and use $K^\Downarrow=B$ along with
Lemma \ref{lem:WDown}(ii).
\end{proof}

\section{The algebra $U_q(\mathfrak{sl}_2)$}
\noindent In the previous sections we  discussed a TD system $\Phi$ of $q$-Racah type.
For the next four sections, we turn our attention to some algebras and their modules. In Section 18 we will return our attention to
$\Phi$. From now until the end of Section 17, fix $0 \not=q\in \mathbb F$ such that $q^4 \not=1$.
 In this section we recall the algebra $U_q(\mathfrak{sl}_2)$ in its equitable presentation. For more information on
 this presentation, see \cite{equit, uawEquit, fduq, lusEquit}.

\begin{definition} \label{def:equit}
\rm (See \cite[Section~2]{equit}.) The algebra $U_q(\mathfrak{sl}_2)$ is defined by generators $x,y^{\pm 1}, z$ and relations
$ y y^{-1} = 1 = y^{-1} y$,
\begin{align}
\frac{q xy - q^{-1} y x}{q-q^{-1}} = 1, 
\qquad \quad
\frac{q yz - q^{-1} zy}{q-q^{-1}} = 1, 
\qquad \quad
\frac{q zx - q^{-1} xz}{q-q^{-1}} = 1.
\label{eq:defREL}
\end{align}
\noindent We call $x, y^{\pm 1}, z$ the {\it equitable generators} of
$U_q(\mathfrak{sl}_2)$.
\end{definition}



\begin{lemma} \label{lem:cas6} {\rm (See \cite[Lemma~2.15]{uawEquit}.)}
The following coincide:
\begin{align*}
& q x + q^{-1} y + q z - q xyz, \qquad \qquad q^{-1}x + q y + q^{-1} z - q^{-1} zyx,
\\
& q y + q^{-1} z + q x - q yzx, \qquad \qquad q^{-1}y + q z + q^{-1} x - q^{-1} xzy,
\\
& q z + q^{-1} x + q y - q zxy, \qquad \qquad q^{-1}z + q x + q^{-1} y- q^{-1} yxz.
\end{align*}
\end{lemma}

\begin{definition} \label{def:casCom} \rm
Let $\bf \Lambda$ denote the common value of the six expressions in Lemma
\ref{lem:cas6}. We call $\bf \Lambda $ the {\it Casimir element} of $U_q(\mathfrak{sl}_2)$.
\end{definition}

\begin{lemma} \label{lem:casCom} 
The element $\bf \Lambda$ generates the center of $U_q(\mathfrak{sl}_2)$. Moreover 
$\lbrace {\bf \Lambda}^i \rbrace_{i \in \mathbb N}$ forms a basis for this center, provided that
$q$ is not a root of unity.
\end{lemma}
\begin{proof} By
\cite[Lemma 2.7,~Proposition 2.18]{jantzen} and \cite[Lemma~2.15]{uawEquit}.
\end{proof}

\noindent Next we discuss the elements $\nu_x$, $\nu_y$, $\nu_z$ of $U_q(\mathfrak{sl}_2)$. Rearranging the relations (\ref{eq:defREL})  we obtain
\begin{align*}
q (1-xy) = q^{-1}(1-yx), \qquad
q (1-yz) = q^{-1}(1-zy), \qquad
q (1-zx) = q^{-1}(1-xz).
\end{align*}
\begin{definition}\label{def:nu}
\rm (See \cite[Definition~3.1]{uawEquit}.)
Define
\begin{align*}
\nu_x &= q(1-yz) = q^{-1}(1-zy),   \\
\nu_y &= q(1-zx) = q^{-1}(1-xz),   \\
\nu_z &= q(1-xy) = q^{-1}(1-yx).
\end{align*}
\end{definition}

\noindent By Definition \ref{def:nu},
\begin{align*}
&xy = 1-q^{-1} \nu_z, \qquad \qquad yx = 1-q \nu_z,     \\
&yz = 1-q^{-1} \nu_x, \qquad \qquad zy = 1-q \nu_x,     \\
&zx = 1-q^{-1} \nu_y, \qquad \qquad  xz = 1-q \nu_y.
\end{align*}
\noindent It follows that
\begin{align*}
\frac{\lbrack x,y\rbrack}{q-q^{-1}} = \nu_z, \qquad \qquad 
\frac{\lbrack y,z\rbrack}{q-q^{-1}} = \nu_x, \qquad \qquad
\frac{\lbrack z,x\rbrack}{q-q^{-1}} = \nu_y.
\end{align*}
\noindent By \cite[Lemma~3.5]{uawEquit},
\begin{align*}
&x \nu_y = q^2 \nu_y x, \qquad \qquad 
y \nu_z = q^2 \nu_z y, \qquad \qquad 
z \nu_x = q^2 \nu_x z, 
\\
&  \nu_z x = q^{2} x \nu_z, \qquad \qquad
 \nu_x y= q^{2} y \nu_x, \qquad \qquad
\nu_y z= q^{2} z \nu_y.
\end{align*}
\noindent By \cite[Lemma~3.7]{uawEquit},
\begin{align*}
\frac{\lbrack x,\nu_x\rbrack}{q-q^{-1}} = y-z, \qquad \qquad 
\frac{\lbrack y,\nu_y\rbrack}{q-q^{-1}} = z-x, \qquad \qquad
\frac{\lbrack z,\nu_z \rbrack}{q-q^{-1}} = x-y.
\end{align*}
\noindent By \cite[Lemma~3.10]{uawEquit},
\begin{align*}
\frac{\lbrack \nu_x, \nu_y \rbrack_q }{q-q^{-1} } = 1-z^2, \qquad \quad 
\frac{\lbrack \nu_y, \nu_z \rbrack_q }{q-q^{-1} } = 1-x^2, \qquad \quad
 \frac{\lbrack \nu_z, \nu_x \rbrack_q }{q-q^{-1} } = 1-y^2. 
\end{align*}

\section{The $q$-tetrahedron algebra $\boxtimes_q$}

\noindent In this section we recall the $q$-tetrahedron algebra
$\boxtimes_q$ and review some of its properties. For more information on this algebra, see 
\cite{qtet, qinv, ItoTerwilliger1, evalTetq, miki, pospart, yy}.

\medskip
\noindent 
Let $\mathbb Z_4=\mathbb Z /4\mathbb Z$ denote the cyclic group of order 4.
\begin{definition} 
\label{def:qtet}
\rm 
\cite[Definition~6.1]{qtet}.
Let $\boxtimes_q$ denote the algebra defined by generators
\begin{align}
\lbrace x_{ij}\;|\;i,j \in 
\mathbb Z_4,
\;\;j-i=1 \;\mbox{\rm {or}}\;
j-i=2\rbrace
\label{eq:gen}
\end{align}
and the following relations:
\begin{enumerate}
\item[\rm (i)] For $i,j \in 
\mathbb Z_4$ such that $j-i=2$,
\begin{align}
 x_{ij}x_{ji} =1.
\label{eq:tet1}
\end{align}
\item[\rm (ii)] For $i,j,k \in 
\mathbb Z_4$ such that $(j-i,k-j)$ is one of $(1,1)$, $(1,2)$, 
$(2,1)$,
\begin{align}
\label{eq:tet2}
\frac{qx_{ij}x_{jk} - q^{-1} x_{jk}x_{ij}}{q-q^{-1}}=1.
\end{align}
\item[\rm (iii)] For $i,j,k,\ell \in 
\mathbb Z_4$ such that $j-i=k-j=\ell-k=1$,
\begin{align}
\label{eq:tet3}
x^3_{ij}x_{k\ell} 
- 
\lbrack 3 \rbrack_q
x^2_{ij}x_{k\ell} x_{ij}
+ 
\lbrack 3 \rbrack_q
x_{ij}x_{k\ell} x^2_{ij}
-
x_{k\ell} x^3_{ij} = 0.
\end{align}
\end{enumerate}
We call 
$\boxtimes_q$ the {\it $q$-tetrahedron algebra}.
The elements
(\ref{eq:gen})
 are called the {\it standard generators} of $\boxtimes_q$.
 The relations
 (\ref{eq:tet3}) are called the {\it $q$-Serre relaions}.
\end{definition}


\noindent  We just gave a presentation of $\boxtimes_q$ by generators and relations. We find it illuminating to describe 
this presentation with a diagram. This diagram is  a directed graph with vertex set
$\mathbb Z_4$.  Each standard generator $x_{ij}$ is represented by a directed arc from vertex $i$ to vertex $j$.  The diagram looks as follows:

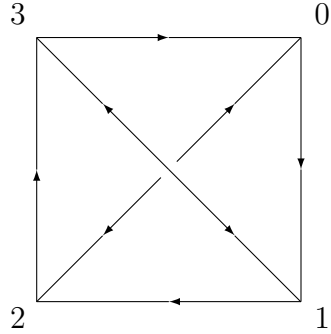
\begin{center}

\begin{picture}(100,80)

\put(-50,50){\vector(1,0){50}}
\put(0,50){\line(1,0){50}}

\put(50,50){\vector(0,-1){50}}
\put(50,0){\line(0,-1){50}}

\put(50,-50){\vector(-1,0){50}}
\put(0,-50){\line(-1,0){50}}

\put(-50,-50){\vector(0,1){50}}
\put(-50,0){\line(0,1){50}}

\put(0,0){\vector(-1,1){25}}
\put(-25,25){\line(-1,1){25}}
\put(0,0){\vector(1,-1){25}}
\put(25,-25){\line(1,-1){25}}

\put(3,3){\vector(1,1){22}}
\put(25,25){\line(1,1){25}}
\put(-3,-3){\vector(-1,-1){22}}
\put(-25,-25){\line(-1,-1){25}}

\put(55,55){$0$}
\put(55,-60){$1$}
\put(-60,-60){$2$}
\put(-60,55){$3$}

\end{picture}
\end{center}
\bigskip

\vspace{2cm}

\noindent  The defining relations for $\boxtimes_q$ can be read off the diagram as follows.
 For any two arcs with the same endpoints and pointing in the  opposite direction, the corresponding
generators are inverses. For any two arcs that create a directed path of length two, the corresponding generators $r$, $s$ satisfy 
\begin{align*} 
\frac{qrs - q^{-1} sr}{q-q^{-1}} = 1.
\end{align*}
For any two arcs that are distinct and parallel (horizontal or vertical), the corresponding generators satisfy the $q$-Serre relations.


\begin{lemma} 
\label{lem:rho}
There exists an 
automorphism
$\rho $ of $\boxtimes_q$ that sends each
standard generator $x_{ij}$ to
$x_{i+1,j+1}$. Moreover
$\rho^4 = 1$.
\end{lemma}
\begin{proof} By Definition \ref{def:qtet}.
\end{proof}

\begin{lemma}
\label{lem:miki}
{\rm (See \cite[Proposition~4.3]{miki}.)}
For $i \in \mathbb Z_4$ there exists an
algebra homomorphism 
$\kappa_i: U_q(\mathfrak{sl}_2)\to \boxtimes_q$ that sends
\begin{eqnarray*}
x \mapsto x_{i+2,i+3},
\qquad 
y \mapsto x_{i+3,i+1},
\qquad
y^{-1} \mapsto x_{i+1,i+3},
\qquad 
z \mapsto x_{i+1,i+2}.
\end{eqnarray*}
This homomorphism is injective.
\end{lemma}




\noindent Recall the Casimir element $\bf \Lambda $
of $U_q(\mathfrak{sl}_2)$, from
Definition \ref{def:casCom}.

\begin{definition} 
\label{def:casi}
\rm For $i \in \mathbb Z_4$ let $\Upsilon_i$ denote
the image of $\bf \Lambda$ under the injection $\kappa_i$ 
from Lemma
\ref{lem:miki}.
\end{definition}

\noindent The elements $\Upsilon_i$ from Definition
\ref{def:casi} are not central in $\boxtimes_q$. However,
we do have the following.

\begin{lemma} \label{lem:CasP}
For $i \in \mathbb Z_4$ the element $\Upsilon_i$ commutes with
each of
\begin{eqnarray*}
x_{i+2,i+3},\qquad
x_{i+3,i+1},\qquad
x_{i+1,i+3},\qquad
x_{i+1,i+2}.
\end{eqnarray*}
\end{lemma}
\begin{proof}
By Lemma
\ref{lem:miki} and since 
$\bf \Lambda$ is central in
 $U_q(\mathfrak{sl}_2)$.
\end{proof}

\section{ The $t$-segregated $\boxtimes_q$-modules}
\noindent We continue to discuss the algebra $\boxtimes_q$ from Definition \ref{def:qtet}.
 In \cite{evalTetq} we introduced a type of $\boxtimes_q$-module, called an evaluation module. 
An evaluation module comes with a nonzero scalar parameter $t$, called the evaluation parameter.
In \cite[Lemmas~9.4,~9.5]{evalTetq} we showed that on a $t$-evaluation $\boxtimes_q$-module the standard generators satisfy
ten  attractive equations involving the commutator map and $t$; these equations are
given in Definition
\ref{def:tseg} 
below.
It turns out that there exist nonevaluation $\boxtimes_q$-modules on which  the ten equations are satisfied;
this fact motivates the following definition.

\begin{definition}
\label{def:tseg} 
\rm
For $0 \not= t\in \mathbb F$, a $\boxtimes_q$-module is called {\it $t$-segregated} whenever it is nonzero, finite-dimensional, and
the following equations hold on the module:
\begin{align}
t(x_{01}-x_{23}) =
\frac{\lbrack x_{30}, x_{12}\rbrack}{q-q^{-1}},
\qquad \qquad
t^{-1}(x_{12}-x_{30}) =
\frac{\lbrack x_{01}, x_{23}\rbrack}{q-q^{-1}}
\label{eq:0123}
\end{align}
and
\begin{align}
&
t (x_{01}-x_{02}) = \frac{\lbrack x_{30}, x_{02}\rbrack}{q-q^{-1}},
\qquad \qquad 
t^{-1} (x_{12}-x_{13}) = 
\frac{\lbrack x_{01}, x_{13}\rbrack}{q-q^{-1}},
\label{eq:four1}
\\
&
t (x_{23}-x_{20}) = \frac{\lbrack x_{12}, x_{20}\rbrack}{q-q^{-1}},
\qquad \qquad 
t^{-1} (x_{30}-x_{31}) = \frac{\lbrack x_{23}, x_{31}\rbrack}{q-q^{-1}}
\label{eq:four2}
\end{align}
and 
\begin{align}
&
t^{-1}(x_{30}-x_{20}) = \frac{\lbrack x_{20},x_{01}\rbrack}{q-q^{-1}},
\qquad \qquad 
t(x_{01}-x_{31}) = \frac{\lbrack x_{31},x_{12}\rbrack}{q-q^{-1}},
\label{eq:four3}
\\
&
t^{-1}(x_{12}-x_{02}) = \frac{\lbrack x_{02},x_{23}\rbrack}{q-q^{-1}},
\qquad \qquad 
t(x_{23}-x_{13}) = \frac{\lbrack x_{13},x_{30}\rbrack}{q-q^{-1}}.
\label{eq:four4}
\end{align}
\end{definition}

\noindent Recall from Definition \ref{def:casi}
the elements $\lbrace \Upsilon_i \rbrace_{i \in \mathbb Z_4}$ in $\boxtimes_q$. We next
consider how these elements act on a $t$-segregated $\boxtimes_q$-module.  Our results
on this topic are given in Lemmas \ref{lem:upsfour}--\ref{lem:aw} below.
These lemmas are proven in \cite[Lemmas~9.11,~9.12,~9.14]{evalTetq} for a $t$-evaluation  $\boxtimes_q$-module;
however the proofs essentially use only the ten equations in
Definition \ref{def:tseg} and consequently apply to every $t$-segregated $\boxtimes_q$-module.

\begin{lemma} 
\label{lem:upsfour} Let $V$ denote a $t$-segregated  $\boxtimes_q$-module.
 Then the action of $\Upsilon_i$ on $V$ is independent of $i \in \mathbb Z_4$.
Denote this common action by $\Upsilon$. Then on $V$,
\begin{align*}
&
\Upsilon = t(x_{01}x_{23}-1)+qx_{30}+q^{-1}x_{12},
 \qquad 
\Upsilon = t^{-1}(x_{12}x_{30}-1)+qx_{01}+q^{-1}x_{23},
\\
&
\Upsilon = t(x_{23}x_{01}-1)+qx_{12}+q^{-1}x_{30},
 \qquad 
\Upsilon = t^{-1}(x_{30}x_{12}-1)+qx_{23}+q^{-1}x_{01}.
\end{align*}
\end{lemma}

\noindent By Lemmas  \ref{lem:CasP}, \ref{lem:upsfour} we find that on a $t$-segregated $\boxtimes_q$-module,
the element $\Upsilon$ commutes with everything in $\boxtimes_q$.

\begin{lemma} 
\label{lem:ups2}
On a $t$-segregated  $\boxtimes_q$-module,
\begin{align*}
&
\Upsilon = (q+q^{-1})x_{30} + t
\biggl(\frac{q x_{01} x_{23} -q^{-1} x_{23} x_{01}}{q-q^{-1}} -1 \biggr),
\\
&
\Upsilon = (q+q^{-1})x_{01} + t^{-1}
\biggl(\frac{q x_{12} x_{30} -q^{-1} x_{30} x_{12}}{q-q^{-1}} -1\biggr),
\\
&
\Upsilon
= (q+q^{-1})x_{12} + t
\biggl(\frac{q x_{23} x_{01} -q^{-1} x_{01} x_{23}}{q-q^{-1}} -1 \biggr),
\\
&
\Upsilon= (q+q^{-1})x_{23} + t^{-1}
\biggl(\frac{q x_{30} x_{12} -q^{-1} x_{12} x_{30}}{q-q^{-1}} -1 \biggr).
\end{align*}
\end{lemma}

\begin{lemma} 
\label{lem:aw}
Let $V$ denote a $t$-segregated $\boxtimes_q$-module. 
Then $x_{01}, x_{23}$
satisfy the following on $V$:
\begin{eqnarray*}
&&
x_{01}^2 x_{23} 
-
(q^2+q^{-2})x_{01} x_{23} x_{01}
+
x_{23} x_{01}^2 \\
&& \qquad \qquad 
=\; -(q-q^{-1})^2 (1+ t^{-1} \Upsilon)x_{01}
+
(q-q^{-1})(q^2-q^{-2})t^{-1},
\\
&&
x_{23}^2 x_{01} 
-
(q^2+q^{-2})x_{23} x_{01} x_{23} 
+
x_{01} x_{23}^2 \\
&& \qquad \qquad 
=\; -(q-q^{-1})^2 (1+ t^{-1} \Upsilon)x_{23}
+
(q-q^{-1})(q^2-q^{-2})t^{-1}.
\end{eqnarray*}
Moreover $x_{12}, x_{30}$
satisfy the following on $V$:
\begin{eqnarray*}
&&
x_{12}^2 x_{30} 
-
(q^2+q^{-2})x_{12} x_{30} x_{12}
+
x_{30} x_{12}^2 \\
&& \qquad \qquad 
=\; -(q-q^{-1})^2 (1+ t \Upsilon)x_{12}
+
(q-q^{-1})(q^2-q^{-2})t,
\\
&&
x_{30}^2 x_{12} 
-
(q^2+q^{-2})x_{30} x_{12} x_{30} 
+
x_{12} x_{30}^2 \\
&& \qquad \qquad 
=\; -(q-q^{-1})^2 (1+ t \Upsilon)x_{30}
+
(q-q^{-1})(q^2-q^{-2})t.
\end{eqnarray*}
\end{lemma}
\noindent We remark that the relations in Lemma \ref{lem:aw} are the Askey-Wilson relations \cite{vidter, zhed}.

\section{How to construct a $t$-segregated $\boxtimes_q$-module}

\noindent We continue to discuss the algebra $\boxtimes_q$ from Definition \ref{def:qtet}. In the previous section, we introduced the concept of a
$t$-segregated $\boxtimes_q$-module. In this section, we show how to construct a $t$-segregated $\boxtimes_q$-module, starting with
a nonzero finite-dimensional $U_q(\mathfrak{sl}_2)$-module and a bit more.
\medskip

\noindent Throughout this section, $V$ denotes a nonzero finite-dimensional $U_q(\mathfrak{sl}_2)$-module.
\begin{assumption} 
\label{prop:extend} Let $0 \not=t \in \mathbb F$. Assume that
there exists an invertible $w \in {\rm End}(V)$ such that on $V$,
\begin{align}
tz-q = w (t-qx), \qquad \qquad 
tz-q^{-1} = (t - q^{-1} x) w.
\label{eq:twocond}
\end{align}
\end{assumption} 
\noindent Under Assumption
\ref{prop:extend}, we will turn $V$ into a $t$-segregated $\boxtimes_q$-module.

\begin{lemma} \label{lem:xwred}
Under Assumption \ref{prop:extend}, the following holds on
 the $U_q(\mathfrak{sl}_2)$-module $V$:
\begin{align*}
x w &= 1 - q t z + q t w,\\
w x &= 1 - q^{-1} t z + q^{-1} t w,
\\
w^{-1} z  &= 1 - q t^{-1} x + q t^{-1} w^{-1},
\\
z w^{-1}  &= 1-q^{-1} t^{-1} x + q^{-1} t^{-1} w^{-1}.
\end{align*}
\end{lemma}
\begin{proof} Use
(\ref{eq:twocond}).
\end{proof}

\noindent Recall the elements $\nu_x$, $\nu_y$, $\nu_z$ of $U_q(\mathfrak{sl}_2)$ from Definition
\ref{def:nu}.
\begin{lemma} \label{lem:nuWred}
Under Assumption \ref{prop:extend}, the following holds on
 the $U_q(\mathfrak{sl}_2)$-module $V$:
 \begin{align*}
 w \nu_z &= q w -qy+t zy - t wy,
 \\
 \nu_z w &= q^{-1} w - q^{-1} y + t yz-t y w,
 \\
 \nu_x w^{-1} &= q w^{-1} - q y + t^{-1} yx - t^{-1} y w^{-1},
 \\
 w^{-1} \nu_x &= q^{-1} w^{-1} -q^{-1} y + t^{-1} xy - t^{-1} w^{-1} y.
 \end{align*}
 \end{lemma}
 \begin{proof} Use 
 Definition \ref{def:nu} and Lemma \ref{lem:xwred}.
 \end{proof}

\begin{proposition} \label{prop:support}
Under Assumption
\ref{prop:extend},  $V$ becomes a $t$-segregated $\boxtimes_q$-module on which the $\boxtimes_q$-generators act as follows:
\bigskip

\centerline{
\begin{tabular}[t]{c|cccccccc}
   {\rm generator} & $x_{01}$ & $x_{12}$ & $x_{23}$ & $x_{30}$ & $x_{02}$ & $x_{13}$ & $x_{20}$ & $x_{31}$
\\
\hline
{\rm action on $V$}  & 
$z$ & $x$ & $y+t^{-1} \nu_z$
&$y+ t \nu_x$ & $y^{-1}$ & $w^{-1}$ & $y$ & $w$
\\
\end{tabular}}
\bigskip
\noindent Moreover $\Upsilon={\bf \Lambda}$ on $V$.
\end{proposition}
\begin{proof} It is trivial to check that the relations 
(\ref{eq:tet1}) hold on $V$. Using 
Lemmas \ref{lem:xwred}, \ref{lem:nuWred} 
and the various relations below Definition \ref{def:nu},
one routinely checks that the relations (\ref{eq:tet2}) and (\ref{eq:0123})--(\ref{eq:four4}) hold on $V$.
Next we check that the $q$-Serre relations (\ref{eq:tet3}) hold on $V$.
 Using Definitions
 \ref{def:casCom}, \ref{def:nu}
 one finds that on $V$,
\begin{align*}
&
{\bf \Lambda} = t(x_{01}x_{23}-1)+qx_{30}+q^{-1}x_{12},
 \qquad 
{\bf \Lambda} = t^{-1}(x_{12}x_{30}-1)+qx_{01}+q^{-1}x_{23},
\\
&
{\bf \Lambda} = t(x_{23}x_{01}-1)+qx_{12}+q^{-1}x_{30},
 \qquad 
{\bf \Lambda} = t^{-1}(x_{30}x_{12}-1)+qx_{23}+q^{-1}x_{01}.
\end{align*}
In other words, the four relations in Lemma \ref{lem:upsfour} hold on $V$ with
$\Upsilon= {\bf \Lambda}$. Using this result, one finds that the relations in
Lemmas
 \ref{lem:ups2}, \ref{lem:aw} hold
on  $V$ with $\Upsilon= {\bf \Lambda}$.
By these comments, the relations (\ref{eq:tet3}) hold on $V$.
The last assertion of the proposition statement follows from the construction.
\end{proof}

\section{The main results}
\noindent In Sections 2--13 we discussed a 
 TD system
 $\Phi=(A;\lbrace E_i\rbrace_{i=0}^d;A^*; \lbrace E^*_i\rbrace_{i=0}^d)$ on $V$ that has $q$-Racah type. In this
 section, we return our attention to $\Phi$. Recall 
 the scalar $a$ from Definition \ref{def:qrac}.
 
 \begin{theorem}\label{thm:one}
 For the above TD system $\Phi$, the underlying vector space $V$ becomes an $a$-segregated
 $\boxtimes_q$-module on which the $\boxtimes_q$-generators act as follows:
\bigskip

\centerline{
\begin{tabular}[t]{c|cccc}
   {\rm generator} & $x_{01}$ & $x_{12}$ & $x_{23}$ & $x_{30}$
\\
\hline
{\rm action on $V$}  & 
$W^{-1} K W$ & $W K^{-1} W^{-1} $ & $Q^{-1}+W \psi W^{-1} $
&$Q^{-1}+ W^{-1} \psi W $
\\
\end{tabular}}
\bigskip

\centerline{
\begin{tabular}[t]{c|cccc}
   {\rm generator} & $x_{02}$ & $x_{13}$ & $x_{20}$ & $x_{31}$
\\
\hline
{\rm action on $V$}  & 
 $Q$ & $K^{-1}$ & $Q^{-1}$ & $K$
\\
\end{tabular}}
\bigskip

\noindent Moreover $\Upsilon=\Lambda$ on $V$.
\end{theorem}
\begin{proof} 
By Proposition \ref{prop:twoET}(i)  and Definition
\ref{def:equit}, the vector space
$V$ becomes a $U_q(\mathfrak{sl}_2)$-module on which 
\begin{align}
x = W K^{-1} W^{-1}, \qquad \quad y = Q^{-1}, \qquad \quad z = W^{-1} K W.
\label{eq:xyzVal}
\end{align}
Define $w=K$, and note that $w$ is invertible.  Also define $t=a$. The element $w$ satisfies (\ref{eq:twocond}) 
by Proposition \ref{thm:WK}.  Assumption \ref{prop:extend}
is now satisfied, so Proposition \ref{prop:support} applies.
Next we show that on $V$,
\begin{align}
\label{eq:needed}
y+t^{-1} \nu_z = Q^{-1} + W \psi W^{-1} , \qquad \quad y+ t\nu_x = Q^{-1} + W^{-1} \psi W.
\end{align}
Since $y=Q^{-1}$ on $V$ and also $t=a$, it suffices to show that on $V$,
\begin{align}
\label{eq:QQ}
a^{-1} \nu_z = W \psi W^{-1} , \qquad \quad a\nu_x = W^{-1} \psi W.
\end{align}
The equation on the left in (\ref{eq:QQ}) is obtained using $\nu_z = q(1-xy)$ with  $x, y$ from
(\ref{eq:xyzVal}), along with $Q^{-1} = W N^{-1} W^{-1}$
and $K^{-1} N^{-1} = I - a q^{-1} \psi$.
The equation on the right in (\ref{eq:QQ}) is obtained using $\nu_x = q(1-yz)$ with  $y, z$ from
(\ref{eq:xyzVal}), along with $Q^{-1} = W^{-1}  M^{-1} W$
and $M^{-1} K = I - a^{-1} q^{-1} \psi$. We have shown that (\ref{eq:needed}) holds on $V$.
It remains to show that $\Upsilon=\Lambda$ on $V$. To do this, by Proposition
\ref{prop:support} it suffices to show that ${\bf \Lambda} = \Lambda $ on $V$. 
On $V$,
\begin{align*}
{\bf \Lambda}  &= q x + q^{-1} y + q z - q xyz
\\
&= q x_{12} + q^{-1} x_{20} + q x_{01} - qx_{12} x_{20} x_{01}.
\end{align*}
By Proposition \ref{prop:wkw}(i) and the construction,  $A=ax_{01} + a^{-1} x_{12}$ on $V$. The generators $x_{01}, x_{12}$ commute with $\bf \Lambda$ on $V$,
so $A$ commutes with $\bf \Lambda$ on $V$.
The element $W$ is a polynomial in $A$, so $W$ commutes with $\bf \Lambda $ on $V$.
We may now argue that on $V$,
\begin{align*}
{\bf \Lambda} &= W(q x_{12} + q^{-1} x_{20} + q x_{01} - qx_{12} x_{20} x_{01})W^{-1}
\\
&= W(q x_{12} (1-x_{20} x_{01}) +  q^{-1} x_{20}+qx_{01}) W^{-1}
\\
&= q W^2 K^{-1} W^{-2} (I-M^{-1} K) + q^{-1} M^{-1} + q K 
\\
& = q(aA-a^2 K) (I-M^{-1} K)  + q^{-1} M^{-1}+ q K
\\
&= (A-aK) \psi + q^{-1} K^{-1} (1-a^{-1} q \psi)+ qK
\\
&= (A-aK - a^{-1}K^{-1})\psi + q K + q^{-1} K^{-1}
\\
& = \Lambda.
\end{align*}
\end{proof}

 \begin{theorem}\label{thm:two}
 For the above TD system $\Phi$, the underlying vector space $V$ becomes an $a^{-1}$-segregated
 $\boxtimes_q$-module on which the $\boxtimes_q$-generators act as follows:
\bigskip

\centerline{
\begin{tabular}[t]{c|cccc}
   {\rm generator} & $x_{01}$ & $x_{12}$ & $x_{23}$ & $x_{30}$
\\
\hline
{\rm action on $V$}  & 
$W^{-1} B W$ & $W B^{-1} W^{-1} $ & $Q^{-1} + W \psi W^{-1} $
&$Q^{-1} + W^{-1} \psi W $
\\
\end{tabular}}
\bigskip

\centerline{
\begin{tabular}[t]{c|cccc}
   {\rm generator} & $x_{02}$ & $x_{13}$ & $x_{20}$ & $x_{31}$
\\
\hline
{\rm action on $V$}  & 
 $Q$ & $B^{-1}$ & $Q^{-1}$ & $B$
\\
\end{tabular}}
\bigskip

\noindent Moreover $\Upsilon=\Lambda$ on $V$.
\end{theorem}
\begin{proof} Apply Theorem
\ref{thm:one} to $\Phi^\Downarrow$, and use $B=K^\Downarrow$ along with
Lemmas \ref{lem:WDown}(ii),  \ref{lem:QDown}, \ref{lem:psiDown}, \ref{lem:LamDown}.
\end{proof}

\begin{note}\rm  
Referring to the tables in Theorems \ref{thm:one}, \ref{thm:two},
for the action of $x_{23}$ and $x_{30}$ an alternative description is given in Proposition \ref{prop:Acom}; see also
Proposition
\ref{prop:Alam}.
\end{note}

\section{Suggestions for future research}

\noindent  In Sections 2--13 and 18 we discussed a TD system 
 $\Phi=(A;\lbrace E_i\rbrace_{i=0}^d;A^*; \lbrace E^*_i\rbrace_{i=0}^d)$ on $V$  that has $q$-Racah type. 
In this section we give some open problems concerning $\Phi$. To motivate
the first problem, we have some comments.
Recall the map $R$ from 
(\ref{eq:R}). Define
\begin{align*}
R^- = W K^{-1} W^{-1} - K^{-1}, \qquad \qquad 
R^+ = W^{-1} K W - K.
\end{align*}
By Proposition \ref{thm:WK},
\begin{align}
R^+ = -a^{-1}  q K R^- = - a^{-1} q^{-1} R^-K.
\label{eq:RPRM}
\end{align}
Therefore
\begin{align}
R^- = -a q^{-1} K^{-1} R^+ = - a q R^+ K^{-1}.
\label{eq:RPRM2}
\end{align}
By (\ref{eq:RPRM}) and (\ref{eq:RPRM2}),
\begin{align*} R^{-} K = q^2 K R^-,
\qquad \qquad
R^+ K = q^2 K R^+.
\end{align*}
Consequently
\begin{align*}
R^{\pm} U_i \subseteq U_{i+1} \qquad \quad (0 \leq i \leq d-1), \quad \qquad R^{\pm} U_d = 0.
\end{align*}
Using Proposition \ref{prop:wkw}(i) and (\ref{eq:R}), 
\begin{align}
\label{eq:RRR}
R=aR^+ + a^{-1} R^-.
\end{align}
Using (\ref{eq:RPRM}) or (\ref{eq:RPRM2}),
\begin{align*} 
R^- R^+ = q^2 R^+ R^-.
\end{align*}
\begin{problem}\rm Investigate the algebraic and combinatorial significance of $R^{\pm}$.
\end{problem}

\noindent On the $\boxtimes_q$-module $V$ in Theorem  \ref{thm:one} 
we have $A = ax_{01}+ a^{-1} x_{12}$, and
on the $\boxtimes_q$-module $V$ in Theorem  \ref{thm:two} 
we have $A =a^{-1}x_{01}+ a x_{12}$. On these modules the value of
$ b^{-1} x_{23} + b x_{30}$ is the same, and it is natural to  guess that this common value is  equal  to $A^*$.  It turns out that this guess is false, but 
it does seem likely that $A^*-b^{-1} x_{23} - b x_{30}$ is important in some way. This motivates the next problem.
\begin{problem}\rm Define
\begin{align*}
\mathcal L = A^*-
b^{-1} (M^{-1} + \psi)
 - b(N^{-1} + \psi) .
\end{align*}
Show that $ \mathcal L K = q^{-2}  K \mathcal L$ and
 $ \mathcal L B = q^{-2}  B\mathcal L$.
 Show that 
 \begin{align*}
 &\mathcal L U_i \subseteq U_{i-1} \qquad \quad (1 \leq i \leq d), \qquad \quad \mathcal L U_0 = 0,
 \\
  &\mathcal L U^\Downarrow_i \subseteq U^\Downarrow_{i-1} \qquad \quad (1 \leq i \leq d), \qquad \quad \mathcal L U^\Downarrow_0 = 0.
 \end{align*}
 Show that $\mathcal L \psi = \psi \mathcal L$. Investigate how $\mathcal L$ is related to $R^\pm$ above.
 \end{problem}

\begin{problem}\label{prob2}
\rm Find a relation involving only $K^{\pm 1}$ and $Q^{\pm 1}$.
\end{problem}

\begin{problem}\label{prob3} \rm
How do $A^*$ and $W A^* W^{-1}$ act on each others eigenspaces? It seems that the pair $A^*$, $W A^* W^{-1}$ is not a TD pair in general.
Find necessary and sufficient conditions on $A, A^*$ for the pair
$A^*$, $W A^* W^{-1}$ to be a TD pair.
\end{problem}


\section{Acknowledgement} The author thanks Kazumasa Nomura for giving this paper a close reading and offering valuable comments.


\bigskip

\noindent Paul Terwilliger \hfil\break
\noindent Department of Mathematics \hfil\break
\noindent University of Wisconsin \hfil\break
\noindent 480 Lincoln Drive \hfil\break
\noindent Madison, WI 53706-1388 USA \hfil\break
\noindent email: {\tt terwilli@math.wisc.edu }\hfil\break

\end{document}